\documentclass[10pt,a4paper,oneside]{article}
\usepackage[english]{babel}
\usepackage{a4wide,amsmath,amsthm,amssymb,url,graphicx,xspace,algorithm,algorithmic,pgf}
\usepackage[latin1]{inputenc}
\usepackage{enumitem}
\usepackage{multirow}
\usepackage{hyperref}
\usepackage{tabularx}
\usepackage[normalem]{ulem}
\usepackage{tikz}
\usepackage{cases}
\usepackage{bm}
\allowdisplaybreaks

\def\F{\mathbb{F}}

\DeclareMathOperator{\PG}{PG}

\DeclareMathOperator{\PGL}{PGL}

\DeclareMathOperator{\PGSp}{PGSp}
\DeclareMathOperator{\PGU}{PGU}

\DeclareMathOperator{\Aut}{Aut}

\newtheorem{theorem}{Theorem}[section]

\newtheorem{lemma}[theorem]{Lemma}

\newtheorem{corollary}[theorem]{Corollary}

\theoremstyle{definition}
\newtheorem{definition}[theorem]{Definition}
\newtheorem{remark}[theorem]{Remark}

\newtheoremstyle{dotless}{}{}{\itshape}{}{\bfseries}{}{ }{}

  \theoremstyle{dotless}

\newcommand{\comments}[1]{}
\newcommand{\gs}[3]{\genfrac{[}{]}{0pt}{}{#1}{#2}_{#3}}

\author{Maarten De Boeck\thanks{Department of Mathematical Sciences, University of Memphis, Dunn Hall, 3725 Norriswood Ave, Memphis, TN 38152, USA. ORCID: 0000-0001-8399-9064. \href{mailto:mdeboeck@memphis.edu}{\url{mdeboeck@memphis.edu}}\newline Department of Mathematics: Algebra and Geometry, Ghent University, Gent, Flanders, Belgium}~ and Geertrui Van de Voorde\thanks{School of Mathematics and Statistics , University of Canterbury, Private Bag 4800, 8140 Christchurch, New Zealand. ORCID: 0000-0002-4957-6911. \href{mailto:geertrui.vandevoorde@canterbury.ac.nz}{\url{geertrui.vandevoorde@canterbury.ac.nz}}}}
\title{Anzahl theorems for trivially intersecting subspaces generating a non-singular subspace I: symplectic and hermitian forms}
\date{}
\begin{document}
\maketitle

\begin{abstract}
    In this paper, we solve a classical counting problem for non-degenerate forms of symplectic and hermitian type defined on a vector space: given a subspace $\pi$, we find the number of non-singular subspaces that are trivially intersecting with $\pi$ and span a non-singular subspace with $\pi$. Lower bounds for the quantity of such pairs where $\pi$ is non-singular were first studied in ``Glasby, Niemeyer, Praeger (Finite Fields Appl., 2022)'', which was later improved in ``Glasby, Ihringer, Mattheus (Des. Codes Cryptogr., 2023)'' and  generalised in ``Glasby, Niemeyer, Praeger (Linear Algebra Appl., 2022)''. In this paper, we derive explicit formulae, which allow us to give the exact proportion and improve the known lower bounds.
\end{abstract}



\section{Introduction}

Counting theorems are at the heart of finite geometry, and so-called \emph{Anzahl theorems} have been studied for geometries defined over finite fields for many decades already. Counting the number of $k$-spaces in a vector space $\F_{q}^{n}$ (equivalently in the projective space $\PG(n-1,q)$) is a standard exercise in an introductory projective geometry class. Another classical result, Lemma \ref{lem:segre}, is due to Segre and counts the number of $k$-spaces in a vector space $\F_{q}^{n}$ trivially intersecting a fixed $j$-space.
\par The non-degenerate quadratic, hermitian and symplectic forms on $\F_{q}^{n}$ define the so-called \emph{classical polar spaces}; they are called quadrics, hermitian and symplectic polar spaces respectively. The {\em subspaces} of these geometries are the totally isotropic subspaces with respect to their defining form. The number of $k$-dimensional subspaces contained in a quadric was determined by Segre (\cite{Segre59}), Ray-Chaudhuri (\cite{Ray}), and Pless \cite{Pless}, who also covered the symplectic case. An important contribution was made by Wan and his students (who also coined the name Anzahl theorem). They counted the number of subspaces with respect to each of the forms, and not only the totally isotropic subspaces, but each of the orbits with respect to the subgroups of $\PGL(n,q)$ stabilising the form up to scalar multiple. Later on, Wan derived also Anzahl theorems for degenerate forms. An overview of his results and other Anzahl theorems can be found in the seminal work \cite{wanbook} and in the survey paper \cite{wansurvey}, which also contains a more detailed history of the Anzahl theorems.
\par The research in this paper has been instigated by the problem posed and discussed in \cite{gim,gnp,gnp2}: given a form on $\F^{n}_{q}$ it asks for the proportion of pairs of non-singular subspaces that are trivially intersecting and span a non-singular subspace. The motivation for this problem comes from computational group theory, more precisely the algorithms for recognising classical groups. For more details we refer to \cite[Section 1.1]{gnp2}. The problem was first discussed in \cite{gnp}, where the special case of complementary pairs of subspaces was discussed, and a lower bound on their proportion was determined. This lower bound was subsequently improved in \cite{gim}. The more general problem, where the subspaces do not necessarily span the entire space, was then discussed in \cite{gnp2} and also here a lower bound was determined.
\par In the current paper we will determine these proportions exactly for the case of hermitian an symplectic forms; we are currently investigating the (more involved) case of the quadratic forms. We will study this problem as an Anzahl problem, that is, we find a hermitian/symplectic analogue for the following result of Segre.

\begin{lemma}\cite[Section 170]{segre}\label{lem:segre}
    The number of $j$-spaces in $\F^{n}_{q}$, trivially intersecting a fixed $k$-space equals $q^{kj}\gs{n-k}{j}{q}$.
\end{lemma}

More precisely, given a fixed subspace $\pi$ of $\F^{n}_{q}$ we will determine the exact number of non-singular subspaces trivially intersecting $\pi$ and spanning a non-singular subspace with it. We will not only do this for the case where $\pi$ is non-singular, but for all types of $\pi$. The results are given in Theorem \ref{th:gammageneralhermitian} (hermitian) and Theorem \ref{th:gammageneralsymplectic} (symplectic).
\par The paper is organised as follows. In Section \ref{sec:preliminaries} we will introduce the necessary background and introduce the functions that we will use to describe the results. In Sections \ref{sec:hermitian} and \ref{sec:symplectic} we will discuss the hermitian and symplectic forms, respectively. In each of the sections, after presenting the main Anzahl theorem, we will determine the proportion that was investigated in \cite{gim,gnp,gnp2}.% \sout{As these formulas are non-trivial to digest we give at the end of each of the two sections a lower bound of the form $1-aq^{-1}-bq^{-2}$, where the constants $a$ and $b$ cannot be improved.}

\section{Preliminaries}\label{sec:preliminaries}

\subsection{Forms}

We first describe the forms we will be working with in this paper.

\begin{definition}
    A \emph{sesquilinear form} on the vector space $\F^{n}_{q}$ is a map $f:\F^{n}_{q}\times \F^{n}_{q}\to\F_{q}$ such that for all $u,v,w\in\F^{n}_{q}$ and all $\lambda,\mu\in\F_{q}$
    \[
        f(u+v,w)=f(u,w)+f(v,w),\qquad f(u,v+w)=f(u,v)+f(u,w)\quad\text{and}\quad f(\lambda v,\mu w)=\lambda\mu^{\theta}f(v,w)
    \]
    for some $\theta\in\Aut(\F_{q})$ with $\theta^2=1$, i.e.~$f$ is linear in the first argument and semi-linear in the second. In case $\theta=1$, then $f$ is called \emph{bilinear}.
    \par A sesquilinear form is called \emph{reflexive} if $f(v,w)=0$ implies $f(w,v)=0$, for all $v,w\in\F^{n}_q$.
\end{definition}

\begin{definition}
    A sesquilinear form $f$ on $\F^{n}_{q}$ is called
    \begin{itemize}
        \item \emph{symmetric} if it is bilinear and $f(v,w)=f(w,v)$ for all $v,w\in\F^{n}_q$;
        \item \emph{symplectic} (or \emph{alternating}) if it is bilinear and $f(v,w)=-f(w,v)$ for all $v,w\in\F^{n}_q$ and $f(v,v)=0$ for all $v\in\F^{n}_q$;
        \item \emph{hermitian} if $\theta\neq1$ and $f(v,w)=f(w,v)^{\theta}$ for all $v,w\in\F^{n}_q$.
    \end{itemize}
\end{definition}

Note that the second condition for symplectic forms is only necessary if $q$ is even. For hermitian forms, the field automorphism $\theta\neq 1$ is necessarily given by $x\mapsto x^{\sqrt{q}}$ and thus hermitian forms only exist when $q$ is a square.

\begin{definition}
    The \emph{radical} of a reflexive sesquilinear form $f$ on $\F^{n}_{q}$ is the set $\{v\in\F^{n}_{q}\mid\forall w\in\F^{n}_{q}:f(v,w)=0\}$. If the radical of $f$ only contains the zero vector, then $f$ is \emph{non-degenerate}.
\end{definition}

It is a classical result in linear algebra that every reflexive sesquilinear form is either symmetric, symplectic or a scalar multiple of a hermitian form. In this article we will only discuss vector spaces equipped with a symplectic or hermitian form. We now introduce their associated polarities.

\begin{definition}
    Given a non-degenerate symplectic or hermitian form $f$ on $\F^{n}_{q}$ we define the corresponding symplectic or hermitian \emph{polarity} $\perp$ as follows:
    \begin{align*}
        \perp:S\to S:\pi\mapsto \{x\in\F^{n}_{q}\mid \forall y\in \pi:f(x,y)=0\}
    \end{align*}
    where $S$ is the set of all subspaces of $\F^{n}_{q}$.
\end{definition}

A polarity on $\F^{n}_{q}$ is an involutory inclusion-reversing permutation of the subspaces. We have that $\dim(\pi)+\dim(\pi^\perp)=n$.

\begin{definition}
    Given a symplectic or hermitian form $f$ on a vector space $V$, the subspaces $\pi$ in $V$ such that the restriction on $\pi$ has an $i$-dimensional radical are called \emph{$i$-singular}. A 0-singular space is also called \emph{non-singular}. If $f$ vanishes on $\pi$ (equivalently, the radical of $\pi$ restricted to $f$ is $\pi$ itself), then $\pi$ is \emph{totally isotropic}.
\end{definition}

The totally isotropic subspaces with respect to a symplectic or a hermitian form are precisely the subspaces $\pi$ such that $\pi\subseteq\pi^{\perp}$ where $\perp$ is the corresponding polarity.

\begin{definition}
    The incidence structure consisting of all totally isotropic subspaces with respect to a symplectic or a hermitian form is a symplectic or hermitian \emph{polar space}. 
\end{definition}

 The symplectic and hermitian polar spaces are two classes of \emph{classical polar spaces}. A classical polar space has a natural embedding in a vector space, and can therefore also be naturally embedded in the corresponding projective space.

Looking at the symplectic or hermitian polar space $\mathcal{P}$ embedded in the projective space $\PG(n-1,q)$, an $i$-singular $j$-space in $\F^{n}_{q}$ corresponds to a $(j-1)$-space of $\PG(n-1,q)$ meeting $\mathcal{P}$ in a cone with an $(i-1)$-dimensional vertex and as base a non-singular polar space of the same type as $\mathcal{P}$ in a $(j-i-1)$-space disjoint to this vertex (where we use vectorial dimension for the subspaces of $\F^{n}_{q}$, and projective dimension for their interpretation in $\PG(n-1,q)$).

In this article the number of $i$-singular $j$-spaces with respect to a symplectic or hermitian form $f$ will typically be denoted by $\alpha_{i,j,n}$.

\subsection{Some useful functions}

We introduce the following functions.

\begin{definition}\label{def:functionsphipsichi}
    For integers $b\geq a\geq0$ we define
    \begin{align*}
        \varphi_{a,b}^{+}(q)&=\prod_{k=a}^{b}\left(q^{k}+(-1)^{k}\right),&\quad\psi_{a,b}^{+}(q)&=\prod_{k=a}^{b}\left(q^{k}+1\right),&\quad\chi_{a,b}(q)&=\prod_{k=a}^{b}\left(q^{2k-1}-1\right),\\
        \varphi_{a,b}^{-}(q)&=\prod_{k=a}^{b}\left(q^{k}-(-1)^{k}\right),&\quad\psi_{a,b}^{-}(q)&=\prod_{k=a}^{b}\left(q^{k}-1\right).
    \end{align*}
    Furthermore, we set $\varphi_{a,a-1}^{+}(q)=\varphi_{a,a-1}^{-}(q)=\psi_{a,a-1}^{+}(q)=\psi_{a,a-1}^{+}(q)=\chi_{a,a-1}(q)=1$ for $a\geq0$; this corresponds to the empty product.
\end{definition}

The classical \emph{Gaussian binomial coefficient} can be described using the function $\psi^{-}$.

\begin{definition}
    For integers $b\geq a\geq 0$ and prime powers $q$ we define
    \begin{align*}
	   \gs{b}{a}{q}=\frac{\psi^{-}_{b-a+1,b}(q)}{\psi^-_{1,a}(q)}\:.
    \end{align*}
    Furthermore, we set $\gs{b}{a}{q}=0$ if $a<0$ or $b<a$.
\end{definition}

It is a standard counting result that $\gs{b}{a}{q}$ equals the number of $a$-dimensional subspaces in $\F^{b}_q$. We now define the following variation of the Gaussian binomial coefficient.

\begin{definition}
    For integers $b\geq a\geq 0$ and prime powers $q$ we define
    \begin{align*}
	   \gs{b}{a}{q}^{-}=\frac{\varphi_{b-a+1,b}^{-}(q)}{\varphi_{1,a}^{-}(q)}\:.
    \end{align*}
    Furthermore, we set $\gs{b}{a}{q}^{-}=0$ if $a<0$ or $b<a$.
\end{definition}

%\sout{Using the Gaussian binomial coefficient we can state an Anzahl theorem of Segre.}

%\begin{lemma}\cite[Section 170]{segre}%\label{lem:segre}
    %\sout{The number of $j$-spaces in $\F^{n}_{q}$, trivially intersecting a fixed $k$-space equals $q^{kj}\gs{n-k}{j}{q}$.}
%\end{lemma}

In the final part of this preliminaries section, we show some upper and lower bounds for the functions introduced in Definition \ref{def:functionsphipsichi}, which we will need when discussing the lower bounds for the proportion in Theorems \ref{th:hermitianrhobound1}, \ref{th:hermitianrhobound2} and \ref{th:symplecticrhobound}.

\begin{lemma}\label{lem:psiminbounds}
    For $a\geq1$ we have
    \[
        \psi^{-}_{1,a}(q)\geq q^{\binom{a+1}{2}}\left(1-\frac{1}{q}-\frac{1}{q^2}+\frac{1}{q^{a+1}}\right)\:,
    \]
    and% for $a\geq2$ we have
    \[
        \psi^{-}_{1,2a}(q)\frac{q^{2a+2}-1}{q^2-1}\geq q^{\binom{2a+2}{2}-1}\left(1-\frac{1}{q}-\frac{1}{q^3}+\frac{q^2-q+1}{q^{2a+3}}\right)\:.
    \]
    %For $a\geq1$ we have $\chi_{1,a}(q)\geq q^{a^2}\left(1-\frac{1}{q}-\frac{1}{q^3}+\frac{1}{q^{2a+1}}\right)$.
\end{lemma}
\begin{proof}
    We prove both results using induction on $a$. The base cases are immediate. We now prove the induction step using the induction hypothesis:
    \begin{align*}
        \psi^{-}_{1,a+1}(q)&=\left(q^{a+1}-1\right)\psi^{-}_{1,a}(q)\\&\geq \left(q^{a+1}-1\right)q^{\binom{a+1}{2}}\left(1-\frac{1}{q}-\frac{1}{q^2}+\frac{1}{q^{a+1}}\right)\\
        %&= q^{\binom{a+1}{2}}\left(q^{a+1}-q^{a}-q^{a-1}+\frac{1}{q}+\frac{1}{q^2}-\frac{1}{q^{a+1}}\right)\\
        &= q^{\binom{a+2}{2}}\left(1-\frac{1}{q}-\frac{1}{q^2}+\frac{1}{q^{a+2}}+\frac{1}{q^{a+3}}-\frac{1}{q^{2a+2}}\right)\\
        &\geq q^{\binom{a+2}{2}}\left(1-\frac{1}{q}-\frac{1}{q^2}+\frac{1}{q^{a+2}}\right)
    \end{align*}
    and
    \begin{align*}
        \psi^{-}_{1,2a+2}(q)\frac{q^{2a+4}-1}{q^2-1}&=\psi^{-}_{1,2a}(q)\left(q^{2a+1}-1\right)\frac{q^{2a+2}-1}{q^2-1}\left(q^{2a+4}-1\right)\\
        &\geq q^{\binom{2a+2}{2}-1}\left(1-\frac{1}{q}-\frac{1}{q^3}+\frac{q^2-q+1}{q^{2a+3}}\right)\left(q^{4a+5}-q^{2a+4}-q^{2a+1}+1\right)\\
        &= q^{\binom{2a+4}{2}-1}\left(1-\frac{1}{q}-\frac{1}{q^3}+\frac{q^{4}+q^{2}+1}{q^{2a+7}}-\frac{q^{6}-q^{5}+q^{4}+q+1}{q^{4a+8}}+\frac{q^2-q+1}{q^{6a+8}}\right)\\
        &\geq q^{\binom{2a+4}{2}-1}\left(1-\frac{1}{q}-\frac{1}{q^3}+\frac{q^2-q+1}{q^{2a+5}}\right)\:,
    \end{align*}
    where the last transition follows from
    \[
        \left(q^{4}+q^{2}+1\right)q^{2a+1}\geq\left(q^{4}+q^{2}+1\right)q^{3}\geq q^{6}-q^{5}+q^{4}+q+1
    \]
    for all prime powers $q\geq2$.
\end{proof}

\begin{lemma}\label{lem:phiminupperbound}
    For integers $a,b$ with $a\geq 2$ and $b\geq0$ we have
    \[
        \varphi^{-}_{b+1,b+a}(q)\leq q^{ab+\binom{a+1}{2}}\left(1+(-1)^{b}q^{-b-a}\frac{q^{a}-(-1)^{a}}{q+1}-q^{-2b-a-1}\right)\;.
    \]
\end{lemma}
\begin{proof}
    We prove this result using induction on $a$. For the base case $a=2$ we find
    \begin{align*}
        \varphi^{-}_{b+1,b+2}(q)&=\left(q^{b+1}-(-1)^{b+1}\right)\left(q^{b+2}-(-1)^{b+2}\right)\\
        &=q^{2b+3}\left(1-(-1)^{b+1}q^{-b-1}-(-1)^{b+2}q^{-b-2}-q^{-2b-3}\right)\\
        &=q^{2b+3}\left(1+(-1)^{b}q^{-b-2}(q-1)-q^{-2b-3}\right)\\
        &\leq q^{2b+\binom{3}{2}}\left(1+(-1)^{b}q^{-b-2}\frac{q^{2}-(-1)^{2}}{q+1}-q^{-2b-3}\right)\;.
    \end{align*}
    We now prove the induction step using the induction hypothesis. For $a\geq 2$ we have:
    \begin{align*}
        \varphi^{-}_{b+1,b+a+1}(q)&=\varphi^{-}_{b+1,b+a}(q)\left(q^{b+a+1}-(-1)^{b+a+1}\right)\\
        &\leq q^{ab+\binom{a+1}{2}}\left(1+(-1)^{b}q^{-b-a}\frac{q^{a}-(-1)^{a}}{q+1}-q^{-2b-a-1}\right)\left(q^{b+a+1}-(-1)^{b+a+1}\right)\\
        &=q^{(a+1)b+\binom{a+2}{2}}\left(1+(-1)^{b}q^{-b-a}\frac{q^{a}-(-1)^{a}}{q+1}-q^{-2b-a-1}\right)\left(1-(-1)^{b+a+1}q^{-b-a-1}\right)\\
        &=q^{(a+1)b+\binom{a+2}{2}}\left(1+(-1)^{b}q^{-b-a}\frac{q^{a}-(-1)^{a}}{q+1}-q^{-2b-a-1}-(-1)^{b+a+1}q^{-b-a-1}\right.\\&\qquad\qquad\qquad\qquad\left.-(-1)^{a+1}q^{-2b-2a-1}\frac{q^{a}-(-1)^{a}}{q+1}+(-1)^{b+a+1}q^{-3b-2a-2}\right)\\
        &=q^{(a+1)b+\binom{a+2}{2}}\left(1+(-1)^{b}q^{-b-a-1}\frac{q^{a+1}-(-1)^{a+1}}{q+1}-q^{-2b-a-1}\right.\\&\qquad\qquad\qquad\qquad\left.-(-1)^{a+1}q^{-2b-2a-1}\frac{q^{a}-(-1)^{a}}{q+1}+(-1)^{b+a+1}q^{-3b-2a-2}\right)\\
        &\leq q^{(a+1)b+\binom{a+2}{2}}\left(1+(-1)^{b}q^{-b-a-1}\frac{q^{a+1}-(-1)^{a+1}}{q+1}-q^{-2b-a-2}\right)\:.
    \end{align*}
    The last transition follows from
    \begin{align*}
        q^{-3b-2a-2}\left((-1)^{a}q^{b+1}\frac{q^{a}-(-1)^{a}}{q+1}+(-1)^{b+a+1}\right)&\leq q^{-3b-2a-2}\left(q^{b+1}\frac{q^{a}-1}{q+1}+1\right)\\
        &=q^{-3b-2a-2}\left(\frac{q^{a+b+1}-q^{b+1}+q+1}{q+1}\right)\\
        &\leq q^{-3b-2a-2}\left(\frac{q^{a+b+1}+1}{q+1}\right)\\
        &\leq q^{-3b-2a-2}\left(\frac{q^{a+b+1}+q^{a+b}}{q+1}\right)\\
        &=q^{-2b-a-2}\\
        &\leq q^{-2b-a-1}-q^{-2b-a-2}\:,
    \end{align*}
    since $q\geq2$.
\end{proof}

\begin{lemma}\label{lem:phiminlowerbound}
    For integers $a,b$ with $a\geq 2$ and $b\geq0$ we have
    \[
        \varphi^{-}_{b+1,b+a}(q)\geq q^{ab+\binom{a+1}{2}}\left(1+(-1)^{b}q^{-b-1}-(-1)^{b}q^{-b-2}-q^{-b-3}-q^{-2b-3}\right)\;.
    \]
\end{lemma}
\begin{proof}
    We prove this result using induction on $a$. The base cases $a=2$ and $a=3$ are immediate. We now prove the induction step using the induction hypothesis. If $b+a+1\geq4$ is odd, then
    \begin{align*}
        \phi^{-}_{b+1,b+a+1}(q)&=\phi^{-}_{b+1,b+a}(q)\left(q^{b+a+1}-(-1)^{b+a+1}\right)=\phi^{-}_{b+1,b+a}(q)q^{b+a+1}\left(1+q^{-b-a-1}\right)\\
        &\geq\phi^{-}_{b+1,b+a}(q)q^{b+a+1}\\
        &\geq q^{(a+1)b+\binom{a+2}{2}}\left(1+(-1)^{b}q^{-b-1}-(-1)^{b}q^{-b-2}-q^{-b-3}-q^{-2b-3}\right)\:.
    \end{align*}
    If $b+a+1\geq5$ is even, then
    \begin{align*}
        \phi^{-}_{b+1,b+a+1}(q)&=\phi^{-}_{b+1,b+a-1}(q)\left(q^{b+a}-(-1)^{b+a}\right)\left(q^{b+a+1}-(-1)^{b+a+1}\right)\\
        &=\phi^{-}_{b+1,b+a-1}(q)q^{2b+2a+1}\left(1+q^{-b-a}-q^{-b-a-1}-q^{-2b-2a-1}\right)\\
        &\geq\phi^{-}_{b+1,b+a-1}(q)q^{2b+2a+1}\\
        &\geq q^{(a+1)b+\binom{a+2}{2}}\left(1+(-1)^{b}q^{-b-1}-(-1)^{b}q^{-b-2}-q^{-b-3}-q^{-2b-3}\right)\:.\qedhere
    \end{align*}
\end{proof}

\begin{lemma}\label{lem:phispluslowerbound}
    For integers $a\geq2$ we have
    \[
        \phi^{+}_{1,a}(q)\geq q^{\binom{a+1}{2}}\left(1-\frac{1}{q}+\frac{1}{q^2}-\frac{2}{q^{3}}\right)\:.
    \]
\end{lemma}
\begin{proof}
    We prove this result using induction on $a$. The base cases $a=2$ and $a=3$ are immediate. We now prove the induction step using the induction hypothesis. If $a+1\geq4$ is even, then
    \begin{align*}
        \phi^{+}_{1,a+1}(q)&=\phi^{+}_{1,a}(q)\left(q^{a+1}+(-1)^{a+1}\right)=\phi^{+}_{1,a}(q)q^{a+1}\left(1+q^{-a-1}\right)\geq\phi^{+}_{1,a}(q)q^{a+1}\\
        &\geq q^{\binom{a+2}{2}}\left(1-\frac{1}{q}+\frac{1}{q^2}-\frac{2}{q^{3}}\right)\:.
    \end{align*}
    If $a+1\geq5$ is odd, then
    \begin{align*}
        \phi^{+}_{1,a+1}(q)&=\phi^{+}_{1,a-1}(q)\left(q^{a}+(-1)^{a}\right)\left(q^{a+1}+(-1)^{a+1}\right)\\
        &=\phi^{+}_{1,a-1}(q)q^{2a+1}\left(1+q^{a}-q^{-a-1}-q^{-2a-1}\right)\\
        &\geq\phi^{+}_{1,a-1}(q)q^{2a+1}\\
        &\geq q^{\binom{a+2}{2}}\left(1-\frac{1}{q}+\frac{1}{q^2}-\frac{2}{q^{3}}\right)\:.\qedhere
    \end{align*}
\end{proof}

\section{The hermitian case}\label{sec:hermitian}
\subsection{Proof of the main result (hermitian)}

We first introduce some notation.

\begin{definition}
	Given a non-degenerate hermitian form $f$ on $\F_{q^{2}}^{n}$, we define $\alpha_{i,j,n}$ as the number of $i$-singular $j$-spaces with respect to $f$.
\end{definition}

The unitary group $\PGU(n,q^2)$, i.e.~the group of all linear maps preserving the non-degenerate hermitian form up to scalar multiple, acts transitively on the $i$-singular $j$-spaces (see \cite[Theorem 2.22]{ht} or \cite[Theorem 5.8]{wanbook}), so the following is well-defined.

\begin{definition}
	Consider a non-degenerate hermitian form $f$ on $\F_{q^{2}}^{n}$. %we define the following numbers.
	\begin{itemize}
		\item Given an $i$-singular $j$-space $\pi$ in $\F_{q^{2}}^{n}$, $\beta_{i,j,n,k}$ is the number of non-singular $k$-spaces $\sigma\supseteq\pi$ in $\F_{q^{2}}^{n}$.% such that the restriction on $\sigma$ is non-singular. In other words, this is the number of projective $(k-1)$-spaces in $\PG(n-1,q^{2})$ meeting $\mathcal{H}(n-1,q^{2})$ in a $\mathcal{H}(k-1,q^{2})$ and containing a fixed $(j-1)$-space that meets $\mathcal{H}(n-1,q^{2})$ in a $\Pi_{i-1}\mathcal{H}(j-i-1,q^{2})$.
		\par For the case of hyperplanes, $k=n-1$, we use the notation $\beta_{i,j,n}=\beta_{i,j,n,n-1}$.
		\item Given an $i$-singular $j$-space $\pi$ in $\F_{q^{2}}^{n}$, $\gamma_{i,j,n,k}$ is the number of non-singular $k$-spaces $\sigma$ in $\F_{q^{2}}^{n}$ such that $\sigma\cap\pi$ is trivial, and $\langle\pi,\sigma\rangle$ is a non-singular $(k+j)$-space.% In other words, this is the number of projective $(k-1)$-spaces in $\PG(n-1,q^{2})$ meeting $\mathcal{H}(n-1,q^{2})$ in a $\mathcal{H}(k-1,q^{2})$, disjoint from a fixed $(j-1)$-space that meets $\mathcal{H}(n-1,q^{2})$ in a $\Pi_{i-1}\mathcal{H}(j-i-1,q^{2})$ and such that the $(k+j-1)$-space spanned by both meets $\mathcal{H}(n-1,q^{2})$ in a $\mathcal{H}(k+j-1,q^{2})$.
		\par For the case $k=n-j$, where the two spaces span $\F_{q^{2}}^{n}$, we use the notation $\gamma_{i,j,n}=\gamma_{i,j,n,n-j}$.
	\end{itemize}
\end{definition}

The main goal of the first subsection is to determine a formula for $\gamma_{i,j,n,k}$. We will first establish the values of $\alpha_{i,j,n}$ and $\beta_{i,j,n}$. The formula for $\alpha_{i,j,n}$ is known in the literature.

\begin{lemma}[{\cite[Theorem 2.23]{ht}, \cite[Theorem 5.19]{wanbook}}]\label{lem:alphahermitian}
	For $0\leq i\leq \min\{j,n-j\}$ and $j\leq n$ we have that
	\[
		\alpha_{i,j,n}=q^{(j-i)(n-j-i)}\frac{\varphi^{-}_{j-i+1,n}(q)}{\varphi^{-}_{1,n-j-i}(q)\psi^{-}_{1,i}(q^2)}\:.%=q^{(j-i)(n-j-i)}\frac{\varphi^{-}_{j-i+1,n}(q)}{\varphi^{-}_{1,n-j-i}(q)\prod^{i}_{m=1}(q^{2m}-1)}
	\]
\end{lemma}

It is not too hard to use a double counting argument to deduce the value for $\beta_{i,j,n,k}$ from that of $\alpha_{i,j,n}$.

\begin{lemma}\label{lem:betahermitian}
	For $0\leq i\leq \min\{j,n-j\}$ and $j+i\leq k\leq n-1$ we have that
	\[
		\beta_{i,j,n,k}=%q^{(n-k)(k-j+i)}\frac{\varphi^{-}_{k-j-i+1,n-j-i}}{\varphi^{-}_{1,n-k}}
		q^{(n-k)(k-j+i)}\gs{n-j-i}{n-k}{q}^{-}\:.
	\]
\end{lemma}
\begin{proof}
	We count the tuples $(\pi,\sigma)$ such that $\pi$ is an $i$-singular $j$-space and $\sigma\supseteq\pi$ is a non-singular $k$-space, in two ways. We find that %We count the tuples $(\pi,\sigma)$ such that $\pi$ is a projective $(j-1)$-space in $\PG(n-1,q^{2})$ meeting $\mathcal{H}(n-1,q^{2})$ in a $\Pi_{i-1}\mathcal{H}(j-i-1,q^{2})$, and $\sigma$ is a $(k-1)$-space in $\PG(n-1,q^{2})$ meeting $\mathcal{H}(n-1,q^{2})$ in a $\mathcal{H}(k-1,q^{2})$, and $\pi\subseteq\sigma$, in two ways. We find that
	\[
		\alpha_{i,j,n}\beta_{i,j,n,k}=\alpha_{0,k,n}\alpha_{i,j,k}\:.
	\]
	Hence, using Lemma \ref{lem:alphahermitian} we find
	\begin{align*}
		\beta_{i,j,n,k}&=\frac{\alpha_{0,k,n}\:\alpha_{i,j,k}}{\alpha_{i,j,n}}\\
		&=q^{k(n-k)}\frac{\varphi^{-}_{k+1,n}(q)}{\varphi^{-}_{1,n-k}(q)}\cdot\frac{q^{(j-i)(k-j-i)}\varphi^{-}_{j-i+1,k}(q)}{\varphi^{-}_{1,k-j-i}(q)\psi^{-}_{1,i}(q^2)}\cdot\frac{\varphi^{-}_{1,n-j-i}(q)\psi^{-}_{1,i}(q^2)}{q^{(j-i)(n-j-i)}\varphi^{-}_{j-i+1,n}(q)}\\
		&=q^{(n-k)(k-j+i)}\frac{\varphi^{-}_{k+1,n}(q)\:\varphi^{-}_{j-i+1,k}(q)\:\varphi^{-}_{1,n-j-i}(q)}{\varphi^{-}_{1,n-k}(q)\:\varphi^{-}_{1,k-j-i}(q)\:\varphi^{-}_{j-i+1,n}(q)}\\
		&=q^{(n-k)(k-j+i)}\frac{\varphi^{-}_{1,n-j-i}(q)}{\varphi^{-}_{1,n-k}(q)\:\varphi^{-}_{1,k-j-i}(q)}\\
		&=q^{(n-k)(k-j+i)}\frac{\varphi^{-}_{k-j-i+1,n-j-i}}{\varphi^{-}_{1,n-k}}\\
		&=q^{(n-k)(k-j+i)}\gs{n-j-i}{n-k}{q}^{-}\:.\qedhere
	\end{align*}
\end{proof}

\begin{corollary}\label{cor:betahermitian}
	For $0\leq i\leq \min\{j,n-j\}$ and $j\leq n-1$ we have that
	\[
		\beta_{i,j,n}=q^{n-j+i-1}\frac{q^{n-i-j}-(-1)^{n-i-j}}{q+1}\:.
	\]
\end{corollary}

Now we derive the formula for $\gamma_{i,j,n}$. In the proof, we first use a double counting argument to find a recursion relation \eqref{eq:hermitianinduction} for $\gamma_{i,j,n}$ based on the values for $\gamma_{i_0,j-1,n-1}$ for $i_0$ differing at most one from $i$. The coefficients are given by functions of $\alpha_{a,b,c}$ and $\beta_{d,e,f}$, which we have determined already. To show that our formula for $\gamma_{i,j,n}$ is correct, we then use induction on $j$.

\begin{theorem}\label{th:gammahermitian}
	For $0\leq i\leq \min\{j,n-j\}$ and $j\leq n-1$ we have that
	\[
	\gamma_{i,j,n}=q^{2j(n-j)-\binom{j+1}{2}}\varphi^{+}_{1,i}(q)\sum_{m=0}^{j-i}(-1)^{m(n-j)}\varphi^{+}_{i+1,j-m}(q)\gs{j-i}{m}{q}^{-}q^{\binom{m}{2}-m(n-j-i)}
	\:.%\gamma_{i,j,n}=q^{(j+i)n-j(j+1)-\binom{i}{2}-ij}\varphi^{+}_{1,i}\sum_{m=0}^{j-i}(-1)^{m(n-j)}\varphi^{+}_{i+1,j-m}(q)\gs{j-i}{m}{q}^{-}q^{(j-i-m)(n-2j)+\binom{j-i-m+1}{2}}
	\]
\end{theorem}
\begin{proof}
	We prove this theorem using induction on $j$. One can see directly that $\gamma_{0,0,n}=1$ for all $n\geq1$. Consider a hermitian form on $\F^{n}_{q^2}$ and a fixed $i$-singular $j$-space $\pi$ with respect to it, with $j\geq1$. We denote the totally isotropic $i$-space of $\pi$ by $\overline{\pi}$. We count the tuples $(\sigma,\tau)$ with $\sigma$ a non-singular hyperplane, and $\tau\subseteq\sigma$ a non-singular $(n-j)$-space disjoint from $\pi$. Note that $\langle\pi,\tau\rangle=\F^{n}_{q^2}$. On the one hand there are $\gamma_{i,j,n}$ choices for $\tau$, and for each of them $\beta_{0,n-j,n}$ corresponding tuples. Now, we look at the non-singular hyperplanes. Each non-singular hyperplane that contains an $(n-j)$-space disjoint from $\pi$ meets $\pi$ in a $(j-1)$-space, so we look at the $(j-1)$-spaces of $\pi$. There are three possibilities. %Consider $\mathcal{H}=\mathcal{H}(n,q^{2})$ and a fixed $(j-1)$-space $\pi$ that meets $\mathcal{H}$ in a $\Pi_{i-1}\mathcal{H}(j-i-1,q^{2})$. We count the tuples $(\sigma,\tau)$ with $\sigma$ a hyperplane meeting $\mathcal{H}$ in an $\mathcal{H}(n-2,q^{2})$, $\tau$ an $(n-j-1)$-space in $\sigma$ meeting $\mathcal{H}$ in an $\mathcal{H}(n-j-1,q^{2})$ disjoint from $\pi$. Note that $\langle\pi,\tau\rangle=\PG(n-1,q^{2})$. On the one hand there are $\gamma_{i,j,n}$ choices for $\tau$, and for each of them $\beta_{0,n-j,n}$ corresponding tuples. Now, we consider the hyperplanes. Each hyperplane that contains an $(n-j-1)$-space disjoint from $\pi$ meets $\pi$ in a $(j-2)$-space, so we look at the $(j-2)$-spaces of $\pi$.
	\begin{itemize}
		\item We first look at the $(j-1)$-spaces through $\overline{\pi}$ that are $(i+1)$-singular. The number of such $(j-1)$-spaces corresponds to the number of 1-singular $(j-i-1)$-spaces with respect to a non-degenerate hermitian form on a $(j-i)$-space, which is $\alpha_{1,j-i-1,j-i}$. Through such a $(j-1)$-space there are $\beta_{i+1,j-1,n}-\beta_{i,j,n}$ non-singular hyperplanes not containing $\pi$. For each of these hyperplanes we have $\gamma_{i+1,j-1,n-1}$ tuples. %The number of $(j-i-2)$-spaces that meet a $\mathcal{H}(j-i-1,q^{2})$ in a $\Pi_0\mathcal{H}(j-i-3,q^{2})$ equals the number of points in a $\mathcal{H}(j-i-1,q^{2})$, namely $\frac{\left(q^{j-i}-(-1)^{j-i}\right)\left(q^{j-i-1}-(-1)^{j-i-1}\right)}{q^{2}-1}$. Any $(j-2)$-space going through the fixed $\Pi_{i-1}$ and meeting the base in a $\Pi_0\mathcal{H}(j-i-3,q^{2})$ meets $\pi$ in a $\Pi_{i}\mathcal{H}(j-i-3,q^{2})$. Through such a $(j-2)$-space there are $\beta_{i+1,j-1,n}-\beta_{i,j,n}$ hyperplanes meeting $\mathcal{H}$ in a $\mathcal{H}(n-2,q^{2})$, not containing $\pi$. For each of these hyperplanes we have $\gamma_{i+1,j-1,n-1}$ tuples.
		\item Secondly we look at the $(j-1)$-spaces through $\overline{\pi}$ that are $i$-singular. The number of such $(j-1)$-spaces corresponds to the number of non-singular $(j-i-1)$-spaces with respect to a non-singular hermitian form on a $(j-i)$-space, which is $\alpha_{0,j-i-1,j-i}$. Through such a $(j-1)$-space there are $\beta_{i,j-1,n}-\beta_{i,j,n}$ non-singular hyperplanes not containing $\pi$. For each of these hyperplanes we have $\gamma_{i,j-1,n-1}$ tuples.
        %The number of $(j-i-2)$-spaces that meet a $\mathcal{H}(j-i-1,q^{2})$ in a $\mathcal{H}(j-i-2,q^{2})$ equals the number of points in a $\PG(j-i-1,q^{2})$ not on a $\mathcal{H}(j-i-1,q^{2})$, namely $\frac{q^{2(j-i)}-1}{q^2-1}-\frac{\left(q^{j-i}-(-1)^{j-i}\right)\left(q^{j-i-1}-(-1)^{j-i-1}\right)}{q^{2}-1}$. Any $(j-2)$-space going through the fixed $\Pi_{i-1}$ and meeting the base in a $\mathcal{H}(j-i-2,q^{2})$ meets $\pi$ in a $\Pi_{i-1}\mathcal{H}(j-i-2,q^{2})$. Through such a $(j-2)$-space there are $\beta_{i,j-1,n}-\beta_{i,j,n}$ hyperplanes meeting $\mathcal{H}$ in a $\mathcal{H}(n-2,q^{2})$, not containing $\pi$. For each of these hyperplanes we have $\gamma_{i,j-1,n-1}$ tuples.
		\item Finally we look at the $(j-1)$-spaces that meet $\overline{\pi}$ in precisely an $(i-1)$-space. Necessarily, they are $(i-1)$-singular. The number of such $(j-1)$-spaces equals $\gs{j}{j-1}{q^{2}}-\gs{j-i}{j-i-1}{q^{2}}$. Through such a $(j-1)$-space there are $\beta_{i-1,j-1,n}-\beta_{i,j,n}$ non-singular hyperplanes not containing $\pi$. For each of these hyperplanes we have $\gamma_{i-1,j-1,n-1}$ tuples.
        %The number of $(j-2)$-spaces not going through the fixed $\Pi_{i-1}$ equals $\frac{q^{2j}-1}{q^2-1}-\frac{q^{2(j-i)}-1}{q^2-1}$. Each of these $(j-2)$-spaces meets $\pi$ in a $\Pi_{i-2}\mathcal{H}(j-i-1,q^{2})$. Through such a $(j-2)$-space there are $\beta_{i-1,j-1,n}-\beta_{i,j,n}$ hyperplanes meeting $\mathcal{H}$ in a $\mathcal{H}(n-2,q^{2})$, not containing $\pi$. For each of these hyperplanes we have $\gamma_{i-1,j-1,n-1}$ tuples.
	\end{itemize}
	We find the following result:
	\begin{align}\label{eq:hermitianinduction}
		\gamma_{i,j,n}\beta_{0,n-j,n}&=\alpha_{1,j-i-1,j-i}\left(\beta_{i+1,j-1,n}-\beta_{i,j,n}\right)\gamma_{i+1,j-1,n-1}\nonumber\\&\quad+\alpha_{0,j-i-1,j-i}\left(\beta_{i,j-1,n}-\beta_{i,j,n}\right)\gamma_{i,j-1,n-1}\nonumber\\&\quad+\left(\gs{j}{j-1}{q^{2}}-\gs{j-i}{j-i-1}{q^{2}}\right)\left(\beta_{i-1,j-1,n}-\beta_{i,j,n}\right)\gamma_{i-1,j-1,n-1}\nonumber
        \\&=\frac{\left(q^{j-i}-(-1)^{j-i}\right)\left(q^{j-i-1}-(-1)^{j-i-1}\right)}{q^{2}-1}\left(\beta_{i+1,j-1,n}-\beta_{i,j,n}\right)\gamma_{i+1,j-1,n-1}\nonumber\\
		&\quad+q^{j-i-1}\left(\frac{q^{j-i}-(-1)^{j-i}}{q+1}\right)\left(\beta_{i,j-1,n}-\beta_{i,j,n}\right)\gamma_{i,j-1,n-1}\nonumber\\
		&\quad+q^{2(j-i)}\left(\frac{q^{2i}-1}{q^2-1}\right)\left(\beta_{i-1,j-1,n}-\beta_{i,j,n}\right)\gamma_{i-1,j-1,n-1}\:,
	\end{align}
	where we used Lemma \ref{lem:alphahermitian}. Note that this equality is also valid if $i=j$ or if $i=j-1$ or $i=0$. Then, only one or two of the cases appear, respectively. But the cases that do not appear, have a factor 0 in Equation \ref{eq:hermitianinduction}. Using Corollary \ref{cor:betahermitian} we know that
	\begin{align*}
		\beta_{i+1,j-1,n}-\beta_{i,j,n}&=q^{n-j+i+1}\frac{q^{n-i-j}-(-1)^{n-i-j}}{q+1}-q^{n-j+i-1}\frac{q^{n-i-j}-(-1)^{n-i-j}}{q+1}\\
		&=q^{n-j+i-1}\left(q^{n-i-j}-(-1)^{n-i-j}\right)\left(q-1\right)\\
		\beta_{i,j-1,n}-\beta_{i,j,n}&=q^{n-j+i}\frac{q^{n-i-j+1}-(-1)^{n-i-j+1}}{q+1}-q^{n-j+i-1}\frac{q^{n-i-j}-(-1)^{n-i-j}}{q+1}\\
		&=q^{n-j+i-1}\left(q^{n-i-j}(q-1)+(-1)^{n-i-j}\right)\\
		\beta_{i-1,j-1,n}-\beta_{i,j,n}&=q^{n-j+i-1}\frac{q^{n-i-j+2}-(-1)^{n-i-j+2}}{q+1}-q^{n-j+i-1}\frac{q^{n-i-j}-(-1)^{n-i-j}}{q+1}\\
		&=q^{n-j+i-1}q^{n-i-j}(q-1)
	\end{align*}
	Using the induction hypothesis we can now rewrite Equation \ref{eq:hermitianinduction} as follows:
	\begin{align*}
	  &\beta_{0,j,n}\gamma_{i,j,n}\\
        &=\frac{\left(q^{j-i}-(-1)^{j-i}\right)\left(q^{j-i-1}-(-1)^{j-i-1}\right)}{q^{2}-1}q^{n-j+i-1}\left(q^{n-i-j}-(-1)^{n-i-j}\right)\left(q-1\right)\\&\qquad\qquad q^{2(j-1)(n-j)-\binom{j}{2}}\varphi^{+}_{1,i+1}(q)\sum_{m=0}^{j-i-2}(-1)^{m(n-j)}\varphi^{+}_{i+2,j-m-1}(q)\gs{j-i-2}{m}{q}^{-}q^{\binom{m}{2}-m(n-j-i-1)}\\%\gamma_{i+1,j-1,n-1}
		&\qquad+q^{j-i-1}\left(\frac{q^{j-i}-(-1)^{j-i}}{q+1}\right)q^{n-j+i-1}\left(q^{n-i-j}(q-1)+(-1)^{n-i-j}\right)\\
		&\qquad\qquad q^{2(j-1)(n-j)-\binom{j}{2}}\varphi^{+}_{1,i}(q)\sum_{m=0}^{j-i-1}(-1)^{m(n-j)}\varphi^{+}_{i+1,j-m-1}(q)\gs{j-i-1}{m}{q}^{-}q^{\binom{m}{2}-m(n-j-i)}\\%\gamma_{i,j-1,n-1}
		&\qquad+q^{2(j-i)}\left(\frac{q^{2i}-1}{q^2-1}\right)q^{n-j+i-1}q^{n-i-j}(q-1)\\
		&\qquad\qquad q^{2(j-1)(n-j)-\binom{j}{2}}\varphi^{+}_{1,i-1}\sum_{m=0}^{j-i}(-1)^{m(n-j)}\varphi^{+}_{i,j-m-1}(q)\gs{j-i}{m}{q}^{-}q^{\binom{m}{2}-m(n-j-i+1)}\\
		&=q^{2(j-1)(n-j)-\binom{j}{2}+n-j+i-1}\frac{\varphi^{+}_{1,i}(q)}{q+1}\\
		&\quad\left(\vphantom{\gs{j-i-2}{m}{q}^{-}}\left(q^{j-i}-(-1)^{j-i}\right)\left(q^{j-i-1}-(-1)^{j-i-1}\right)\left(q^{n-i-j}-(-1)^{n-i-j}\right)\right.\\&\qquad\qquad \sum_{m=0}^{j-i-2}(-1)^{m(n-j)}\varphi^{+}_{i+1,j-m-1}(q)\gs{j-i-2}{m}{q}^{-}q^{\binom{m}{2}-m(n-j-i-1)}\\
		&\qquad+q^{j-i-1}\left(q^{j-i}-(-1)^{j-i}\right)\left(q^{n-i-j}(q-1)+(-1)^{n-i-j}\right)\\
		&\qquad\qquad \sum_{m=0}^{j-i-1}(-1)^{m(n-j)}\varphi^{+}_{i+1,j-m-1}(q)\gs{j-i-1}{m}{q}^{-}q^{\binom{m}{2}-m(n-j-i)}\\
		&\qquad+\left.\left(q^{2i}-1\right)q^{n+j-3i} \sum_{m=0}^{j-i}(-1)^{m(n-j)}\varphi^{+}_{i+1,j-m-1}(q)\gs{j-i}{m}{q}^{-}q^{\binom{m}{2}-m(n-j-i+1)}\right)\\	
		&=q^{2(j-1)(n-j)-\binom{j}{2}+n-j+i-1}\frac{\varphi^{+}_{1,i}(q)}{q+1}\\
		&\quad\left(\vphantom{\gs{j-i-2}{m}{q}^{-}}\left(q^{j-i}-(-1)^{j-i}\right)\left(q^{n-i-j}-(-1)^{n-i-j}\right)\right.\\&\qquad\quad \sum_{m=0}^{j-i-1}(-1)^{m(n-j)}\varphi^{+}_{i+1,j-m-1}(q)\gs{j-i-1}{m}{q}^{-}\left(q^{j-i-m-1}-(-1)^{j-i-m-1}\right)q^{\binom{m}{2}-m(n-j-i-1)}\\
		&\qquad+q^{j-i-1}\left(q^{j-i}-(-1)^{j-i}\right)\left(q^{n-i-j+1}-\left(q^{n-i-j}-(-1)^{n-i-j}\right)\right)\\
		&\qquad\quad \sum_{m=0}^{j-i-1}(-1)^{m(n-j)}\varphi^{+}_{i+1,j-m-1}(q)\gs{j-i-1}{m}{q}^{-}q^{\binom{m}{2}-m(n-j-i)}\\
		&\qquad+\left.\left(q^{2i}-1\right)q^{n+j-3i} \sum_{m=0}^{j-i}(-1)^{m(n-j)}\varphi^{+}_{i+1,j-m-1}(q)\gs{j-i}{m}{q}^{-}q^{\binom{m}{2}-m(n-j-i+1)}\right)\\
		&=q^{2(j-1)(n-j)-\binom{j}{2}+n-j+i-1}\frac{\varphi^{+}_{1,i}(q)}{q+1}\\
		&\quad\left(\vphantom{\gs{j-i-2}{m}{q}^{-}}q^{j-i-1}\left(q^{j-i}-(-1)^{j-i}\right)\left(q^{n-i-j}-(-1)^{n-i-j}\right)\right.\\&\qquad\quad \sum_{m=0}^{j-i-1}(-1)^{m(n-j)}\varphi^{+}_{i+1,j-m-1}(q)\gs{j-i-1}{m}{q}^{-}\left(q^{-m}-q^{-m}\right)q^{\binom{m}{2}-m(n-j-i-1)}\\
		&\qquad+(-1)^{j-i}\left(q^{j-i}-(-1)^{j-i}\right)\left(q^{n-i-j}-(-1)^{n-i-j}\right)\\&\qquad\quad \sum_{m=0}^{j-i-1}(-1)^{m(n-j-1)}\varphi^{+}_{i+1,j-m-1}(q)\gs{j-i-1}{m}{q}^{-}q^{\binom{m}{2}-m(n-j-i-1)}\\
		&\qquad+q^{n-2i}\left(q^{j-i}-(-1)^{j-i}\right)\sum_{m=0}^{j-i-1}(-1)^{m(n-j)}\varphi^{+}_{i+1,j-m-1}(q)\gs{j-i-1}{m}{q}^{-}q^{\binom{m}{2}-m(n-j-i)}\\
		&\qquad+\left.\left(q^{2i}-1\right)q^{n+j-3i} \sum_{m=0}^{j-i}(-1)^{m(n-j)}\varphi^{+}_{i+1,j-m-1}(q)\gs{j-i}{m}{q}^{-}q^{\binom{m}{2}-m(n-j-i+1)}\right)\\
		&=q^{2(j-1)(n-j)-\binom{j}{2}+n-j+i-1}\frac{\varphi^{+}_{1,i}(q)}{q+1}\\
		&\quad\left(\vphantom{\gs{j-i-2}{m}{q}^{-}}(-1)^{j-i}\left(q^{n-i-j}-(-1)^{n-i-j}\right)\right.\\&\qquad\quad \sum_{m=0}^{j-i-1}(-1)^{m(n-j-1)}\varphi^{+}_{i+1,j-m-1}(q)\gs{j-i}{m}{q}^{-}\left(q^{j-i-m}-(-1)^{j-i-m}\right)q^{\binom{m}{2}-m(n-j-i-1)}\\
		&\qquad+q^{n-2i}\\
		&\qquad\quad \sum_{m=0}^{j-i-1}(-1)^{m(n-j)}\varphi^{+}_{i+1,j-m-1}(q)\gs{j-i}{m}{q}^{-}\left(q^{j-i-m}-(-1)^{j-i-m}\right)q^{\binom{m}{2}-m(n-j-i)}\\
		&\qquad+\left.\left(q^{2i}-1\right)q^{n+j-3i} \sum_{m=0}^{j-i}(-1)^{m(n-j)}\varphi^{+}_{i+1,j-m-1}(q)\gs{j-i}{m}{q}^{-}q^{\binom{m}{2}-m(n-j-i+1)}\right)\\
		&=q^{2(j-1)(n-j)-\binom{j}{2}+n-j+i-1}\frac{\varphi^{+}_{1,i}(q)}{q+1}\\
		&\quad\left(\vphantom{\gs{j-i-2}{m}{q}^{-}}(-1)^{j-i}\left(q^{n-i-j}-(-1)^{n-i-j}\right)q^{j-i}\right.\\&\qquad\quad \sum_{m=0}^{j-i-1}(-1)^{m(n-j-1)}\varphi^{+}_{i+1,j-m-1}(q)\gs{j-i}{m}{q}^{-}q^{\binom{m}{2}-m(n-j-i)}\\
		&\qquad-\left(q^{n-i-j}-(-1)^{n-i-j}\right) \sum_{m=0}^{j-i-1}(-1)^{m(n-j)}\varphi^{+}_{i+1,j-m-1}(q)\gs{j-i}{m}{q}^{-}q^{\binom{m}{2}-m(n-j-i-1)}\\
		&\qquad+q^{n+j-3i}\sum_{m=0}^{j-i-1}(-1)^{m(n-j)}\varphi^{+}_{i+1,j-m-1}(q)\gs{j-i}{m}{q}^{-}q^{\binom{m}{2}-m(n-j-i+1)}\\
		&\qquad-(-1)^{j-i}q^{n-2i}\sum_{m=0}^{j-i-1}(-1)^{m(n-j-1)}\varphi^{+}_{i+1,j-m-1}(q)\gs{j-i}{m}{q}^{-}q^{\binom{m}{2}-m(n-j-i)}\\
		&\qquad+\left.\left(q^{2i}-1\right)q^{n+j-3i} \sum_{m=0}^{j-i}(-1)^{m(n-j)}\varphi^{+}_{i+1,j-m-1}(q)\gs{j-i}{m}{q}^{-}q^{\binom{m}{2}-m(n-j-i+1)}\right)\\
		&=q^{2(j-1)(n-j)-\binom{j}{2}+n-j+i-1}\frac{\varphi^{+}_{1,i}(q)}{q+1}\\
		&\quad\left(-(-1)^{n}q^{j-i}\sum_{m=0}^{j-i-1}(-1)^{m(n-j-1)}\varphi^{+}_{i+1,j-m-1}(q)\gs{j-i}{m}{q}^{-}q^{\binom{m}{2}-m(n-j-i)}\right.\\
		&\qquad-q^{n-i-j}\sum_{m=0}^{j-i-1}(-1)^{m(n-j)}\varphi^{+}_{i+1,j-m-1}(q)\gs{j-i}{m}{q}^{-}q^{\binom{m}{2}-m(n-j-i-1)}\\
		&\qquad+(-1)^{n-i-j} \sum_{m=0}^{j-i-1}(-1)^{m(n-j)}\varphi^{+}_{i+1,j-m-1}(q)\gs{j-i}{m}{q}^{-}q^{\binom{m}{2}-m(n-j-i-1)}\\
		&\qquad+\left.q^{n+j-i} \sum_{m=0}^{j-i}(-1)^{m(n-j)}\varphi^{+}_{i+1,j-m-1}(q)\gs{j-i}{m}{q}^{-}q^{\binom{m}{2}-m(n-j-i+1)}\right)\\
		&=q^{2(j-1)(n-j)-\binom{j}{2}+n-j+i-1}\frac{\varphi^{+}_{1,i}(q)}{q+1}\\
		&\quad\left(-(-1)^{n}\sum_{m=0}^{j-i-1}(-1)^{m(n-j-1)}\varphi^{+}_{i+1,j-m-1}(q)\gs{j-i}{m}{q}^{-}q^{\binom{m}{2}-m(n-j-i-1)}\left(q^{j-i-m}-(-1)^{j-i-m}\right)\right.\\
		&\qquad+\left.q^{n-i-j} \sum_{m=0}^{j-i}(-1)^{m(n-j)}\varphi^{+}_{i+1,j-m-1}(q)\gs{j-i}{m}{q}^{-}q^{\binom{m}{2}-m(n-j-i-1)}\left(q^{2(j-m)}-1\right)\right)\\
		&=q^{2(j-1)(n-j)-\binom{j}{2}+n-j+i-1}\frac{\varphi^{+}_{1,i}(q)}{q+1}\\
		&\quad\left((-1)^{n-1}\sum_{m=0}^{j-i-1}(-1)^{m(n-j-1)}\varphi^{+}_{i+1,j-m-1}(q)\gs{j-i}{m+1}{q}^{-}q^{\binom{m}{2}-m(n-j-i-1)}\left(q^{m+1}-(-1)^{m+1}\right)\right.\\
		&\qquad+q^{n-i-j}\sum_{m=0}^{j-i}(-1)^{m(n-j)}\varphi^{+}_{i+1,j-m-1}(q)\gs{j-i}{m}{q}^{-}q^{\binom{m}{2}-m(n-j-i-1)}\\&\qquad\qquad\qquad\qquad\qquad\left. \left(q^{j-m}+(-1)^{j-m}\right)\left(q^{j-m}-(-1)^{j-m}\right)\vphantom{\gs{j-i}{m}{q}^{-}}\right)\\
		&=q^{2(j-1)(n-j)-\binom{j}{2}+n-j+i-1}\frac{\varphi^{+}_{1,i}(q)}{q+1}\\
		&\quad\left((-1)^{n-1}\sum_{m=1}^{j-i}(-1)^{(m-1)(n-j-1)}\varphi^{+}_{i+1,j-m}(q)\gs{j-i}{m}{q}^{-}q^{\binom{m-1}{2}-(m-1)(n-j-i-1)}\left(q^{m}-(-1)^{m}\right)\right.\\
		&\qquad+\left.q^{n-i-j} \sum_{m=0}^{j-i}(-1)^{m(n-j)}\varphi^{+}_{i+1,j-m}(q)\gs{j-i}{m}{q}^{-}q^{\binom{m}{2}-m(n-j-i-1)}\left(q^{j-m}-(-1)^{j-m}\right)\right)\\
		&=q^{2(j-1)(n-j)-\binom{j}{2}+n-j+i-1}\frac{\varphi^{+}_{1,i}(q)}{q+1}\\
		&\quad\left((-1)^{j}q^{n-j-i}\sum_{m=1}^{j-i}(-1)^{m(n-j-1)}\varphi^{+}_{i+1,j-m}(q)\gs{j-i}{m}{q}^{-}q^{\binom{m}{2}-m(n-j-i-1)}\right.\\
		&\qquad-(-1)^{j}q^{n-j-i}\sum_{m=1}^{j-i}(-1)^{m(n-j)}\varphi^{+}_{i+1,j-m}(q)\gs{j-i}{m}{q}^{-}q^{\binom{m}{2}-m(n-j-i)}\\
		&\qquad+q^{n-i} \sum_{m=0}^{j-i}(-1)^{m(n-j)}\varphi^{+}_{i+1,j-m}(q)\gs{j-i}{m}{q}^{-}q^{\binom{m}{2}-m(n-j-i)}\\
		&\qquad-(-1)^{j}\left.q^{n-i-j} \sum_{m=0}^{j-i}(-1)^{m(n-j-1)}\varphi^{+}_{i+1,j-m}(q)\gs{j-i}{m}{q}^{-}q^{\binom{m}{2}-m(n-j-i-1)}\right)\\
		&=q^{2j(n-j)-\binom{j}{2}-1}\frac{\varphi^{+}_{1,i}(q)}{q+1}\\
		&\quad\left(\left(q^{j}-(-1)^{j}\right)\sum_{m=1}^{j-i}(-1)^{m(n-j)}\varphi^{+}_{i+1,j-m}(q)\gs{j-i}{m}{q}^{-}q^{\binom{m}{2}-m(n-j-i)}\right.\\
		&\left.\qquad+q^{j}\varphi^{+}_{i+1,j}(q)-(-1)^{j} \varphi^{+}_{i+1,j}(q)\vphantom{\sum_{m=1}^{j-i}\gs{j-i}{m}{q}^{-}}\right)\\
		&=q^{j-1}\left(q^{j}-(-1)^{j}\right)q^{2j(n-j)-\binom{j+1}{2}}\frac{\varphi^{+}_{1,i}(q)}{q+1}\\
		&\qquad\sum_{m=0}^{j-i}(-1)^{m(n-j)}\varphi^{+}_{i+1,j-m}(q)\gs{j-i}{m}{q}^{-}q^{\binom{m}{2}-m(n-j-i)}+\varphi^{+}_{i+1,j}(q)
	\end{align*}
	from which the formula for $\gamma_{i,j,n}$ follows, since the coefficient of $\gamma_{i,j,n}$ in \eqref{eq:hermitianinduction} is non-zero:
	\[
		\beta_{0,n-j,n}=q^{j-1}\frac{q^{j}-(-1)^{j}}{q+1}\neq0\:.\qedhere
	\]
\end{proof}

\begin{theorem}\label{th:gammageneralhermitian}
	For $0\leq i\leq \min\{j,n-j\}$, $j\leq n-1$ and $k\leq n-j$, we have that
	\begin{align*}
		\gamma_{i,j,n,k}&=%\beta_{i,j,n,k}\gamma_{i,j,k}\\
		%&=q^{(n-k)(k-j+i)+2j(k-j)-\binom{j+1}{2}}\frac{\varphi^{+}_{1,i}(q)\varphi^{-}_{k-j-i+1,n-j-i}}{\varphi^{-}_{1,n-k}}\sum_{m=0}^{j-i}(-1)^{m(k-j)}\varphi^{+}_{i+1,j-m}(q)\gs{j-i}{m}{q}^{-}q^{\binom{m}{2}-m(k-j-i)}
		\beta_{i,j,n,k+j}\gamma_{i,j,k+j}\\
		&=q^{(n-k-j)(k+i)+2jk-\binom{j+1}{2}}\gs{n-j-i}{n-k-j}{q}^{-}\varphi^{+}_{1,i}(q)\sum_{m=0}^{j-i}(-1)^{mk}\varphi^{+}_{i+1,j-m}(q)\gs{j-i}{m}{q}^{-}q^{\binom{m}{2}-m(k-i)}
	\end{align*}
\end{theorem}
\begin{proof}
	Consider a fixed $i$-singular $j$-space. Now, we count in two ways the pairs $(\sigma,\tau)$, where $\sigma$ is non-singular $k$-space, disjoint from $\pi$, and $\tau$ is a non-singular $(k+j)$-space and $\pi,\sigma\subseteq\tau$. The equality $\gamma_{i,j,n,k}=\beta_{i,j,n,k+j}\gamma_{i,j,k+j}$ immediately follows. The second part of the result then follows from Lemma \ref{lem:betahermitian} and Theorem \ref{th:gammahermitian}.
\end{proof}

\subsection{The proportion of non-singular trivially intersecting subspaces spanning a non-singular space}

In this subsection, we look at the proportion that motivated this research.

\begin{definition}\label{def:rhohermitian}
    Given a non-degenerate hermitian form on $\F^{n}_{q^2}$ and integers $j,k$ with $0\leq j,k\leq n-1$ and $j+k\leq n$, let $\mathcal{S}_{j,k}$ be the set of pairs $(\pi,\pi')$ with $\dim(\pi)=j$ and $\dim(\pi')=k$ and both $\pi$ and $\pi'$ non-singular. Let $\mathcal{T}_{j,k}$ be the subset of $\mathcal{S}_{j,k}$ with pairs $(\pi,\pi')$ such that $\dim(\langle\pi,\pi'\rangle)=j+k$ and $\langle\pi,\pi'\rangle$ non-singular. The proportion $\frac{|\mathcal{T}_{j,k}|}{|\mathcal{S}_{j,k}|}$ is denoted by $\rho_{j,k,n}$.
\end{definition}

Note that by definition $\rho_{j,k,n}=\rho_{k,j,n}$.

\begin{theorem}\label{th:rhohermitian}
	For integers $j,k,n$ with $0\leq j,k\leq n-1$ and $j+k\leq n$, we have
    \[
        \rho_{j,k,n}=q^{jk-\binom{j+1}{2}}\frac{\varphi^{-}_{n-j-k+1,n-k}(q)}{\varphi^{-}_{n-j+1,n}(q)}\sum_{m=0}^{j}(-1)^{mk}\varphi^{+}_{1,j-m}(q)\gs{j}{m}{q}^{-}q^{\binom{m}{2}-mk}
    \]
\end{theorem}
\begin{proof}
    Since the unitary group $\PGU(n,q^{2})$ acts transitively on the non-singular $j$-spaces we have immediately that $\rho_{j,k,n}=\frac{\gamma_{0,j,n,k}}{\alpha_{0,k,n}}$. From Lemma \ref{lem:alphahermitian} and Theorem \ref{th:gammageneralhermitian} we get
    \begin{align*}
        \rho_{j,k,n}&=\frac{q^{(n-k+j)k-\binom{j+1}{2}}\gs{n-j}{k}{q}^{-}\sum_{m=0}^{j}(-1)^{mk}\varphi^{+}_{1,j-m}(q)\gs{j}{m}{q}^{-}q^{\binom{m}{2}-mk}}{q^{k(n-k)}\frac{\varphi^{-}_{k+1,n}(q)}{\varphi^{-}_{1,n-k}(q)}}\\
        &=q^{jk-\binom{j+1}{2}}\frac{\varphi^{-}_{n-j-k+1,n-j}(q)}{\varphi^{-}_{n-k+1,n}(q)}\sum_{m=0}^{j}(-1)^{mk}\varphi^{+}_{1,j-m}(q)\gs{j}{m}{q}^{-}q^{\binom{m}{2}-mk}\\
        &=q^{jk-\binom{j+1}{2}}\frac{\varphi^{-}_{n-j-k+1,n-k}(q)}{\varphi^{-}_{n-j+1,n}(q)}\sum_{m=0}^{j}(-1)^{mk}\varphi^{+}_{1,j-m}(q)\gs{j}{m}{q}^{-}q^{\binom{m}{2}-mk}\;.\qedhere
    \end{align*}
\end{proof}

\begin{remark}
    In \cite{gnp} it was proven that $\rho_{j,k,j+k}\geq 1-\frac{9}{5}q^{-2}$, which was improved to $\rho_{j,k,j+k}\geq 1-\frac{3}{2}q^{-2}$ in \cite{gim}  (unless $(j,k,q)=(1,1,2)$). In \cite{gnp2} it was proven that $\rho_{j,k,n}\geq 1-\frac{43}{25}q^{-1}$ if $n\geq j+k+1$.  The previous theorem improves these results by giving the exact value of the proportion.
\end{remark}

Although we have an exact expression for the ratio of non-singular pairs that are disjoint, we will present a lower bound for it, as to get a better feeling for its asymptotic behavior.
\par It is immediate (both from Definition \ref{def:rhohermitian} and Theorem \ref{th:rhohermitian}) that $\rho_{0,k,n}=1$, and by symmetry $\rho_{j,0,n}=1$. We now look at the other cases. We distinguish between $n=j+k$ and $n\geq j+k+1$, but first prove a lemma.

\begin{lemma}\label{lem:boundingthesummands}
    For integers $j$, $k$ and $m$ with $0\leq m\leq j-1$ and $2\leq j\leq k$ we have
    \begin{align*}
        \varphi^{+}_{1,j-m}(q)\gs{j}{m}{q}^{-}q^{\binom{m}{2}-mk}\geq\varphi^{+}_{1,j-m-1}(q)\gs{j}{m+1}{q}^{-}q^{\binom{m+1}{2}-(m+1)k}\:.
    \end{align*}
\end{lemma}
\begin{proof}
    It is immediate that
    \begin{align*}
        &\varphi^{+}_{1,j-m}(q)\gs{j}{m}{q}^{-}q^{\binom{m}{2}-mk}\geq\varphi^{+}_{1,j-m-1}(q)\gs{j}{m+1}{q}^{-}q^{\binom{m+1}{2}-(m+1)k}\\
        \Leftrightarrow\qquad&\left(q^{j-m}+(-1)^{j-m}\right)\left(q^{m+1}-(-1)^{m+1}\right)\geq q^{m-k}(q^{j-m}-(-1)^{j-m})\:,
    \end{align*}
    which is true since
    \begin{align*}
        \left(q^{j-m}+(-1)^{j-m}\right)\left(q^{m+1}-(-1)^{m+1}\right)&\geq q^{j+1}-q^{m+1}-q^{j-m}+1\\
        &\geq q^{3}-q^{2}-q+1\\
        &\geq 1+q^{-1}\\
        &\geq q^{j-k}+q^{m-k}\\
        &\geq q^{m-k}\left(q^{j-m}+1\right)\\
        &\geq q^{m-k}(q^{j-m}-(-1)^{j-m})\:.\qedhere
    \end{align*}
\end{proof}

\begin{theorem}\label{th:hermitianrhobound1}
    Let $j$ and $k$ be integers with $j,k\geq 1$. We have
    \begin{align*}
        \rho_{1,k,k+1}&=1-\frac{1}{q^2}+(-1)^{k}\frac{q+1}{q^{2}\left(q^{k+1}+(-1)^{k}\right)}\:,\\
        \rho_{j,1,j+1}&=1-\frac{1}{q^2}+(-1)^{j}\frac{q+1}{q^{2}\left(q^{j+1}+(-1)^{j}\right)}
    \end{align*}
    and
    \[
        \rho_{j,k,j+k}\geq 1-\frac{1}{q^{2}}-\frac{1}{q^4}-\frac{1}{q^{k+j+1}}+\max\left\{\frac{1}{q^{2j+2}},\frac{1}{q^{2k+2}}\right\}
    \]
    if $j,k\geq2$.
\end{theorem}
\begin{proof}
    From Theorem \ref{th:rhohermitian} we have immediately
    \begin{align*}
        \rho_{j,k,j+k}%&=q^{jk-\binom{j+1}{2}}\frac{\varphi^{-}_{1,k}(q)}{\varphi^{-}_{j+1,j+k}(q)}\sum_{m=0}^{j}(-1)^{mk}\varphi^{+}_{1,j-m}(q)\gs{j}{m}{q}^{-}q^{\binom{m}{2}-mk}\\
        &=q^{jk-\binom{j+1}{2}}\frac{\varphi^{-}_{1,j}(q)}{\varphi^{-}_{k+1,j+k}(q)}\sum_{m=0}^{j}(-1)^{mk}\varphi^{+}_{1,j-m}(q)\gs{j}{m}{q}^{-}q^{\binom{m}{2}-mk}\:.
    \end{align*}
    The formula for the case $j=1$ follows immediately, and by symmetry also the formula for the case $k=1$. We now look at the the case $j,k\geq2$. Because of the symmetry we may assume without loss of generality that $j\leq k$. From Lemma \ref{lem:boundingthesummands} it follows that
    \begin{align*}
        \varphi^{+}_{1,j-m}(q)\gs{j}{m}{q}^{-}q^{\binom{m}{2}-mk}\geq\varphi^{+}_{1,j-m-1}(q)\gs{j}{m+1}{q}^{-}q^{\binom{m+1}{2}-(m+1)k}\:.
    \end{align*}
    \comments{We note that for $0\leq m\leq j-1$
    \begin{align*}
        &\varphi^{+}_{1,j-m}(q)\gs{j}{m}{q}^{-}q^{\binom{m}{2}-mk}\geq\varphi^{+}_{1,j-m-1}(q)\gs{j}{m+1}{q}^{-}q^{\binom{m+1}{2}-(m+1)k}\\
        \Leftrightarrow\qquad&\left(q^{j-m}+(-1)^{j-m}\right)\left(q^{m+1}-(-1)^{m+1}\right)\geq q^{m-k}(q^{j-m}-(-1)^{j-m})\:,
    \end{align*}
    which is true since
    \begin{align*}
        \left(q^{j-m}+(-1)^{j-m}\right)\left(q^{m+1}-(-1)^{m+1}\right)&\geq q^{j+1}-q^{m+1}-q^{j-m}+1\\
        &\geq q^{j+1}-q^{j}\\
        &\geq q^{j-1}+q^{j-2}\\
        &\geq q^{-k}\left(q^{j}+q^{m}\right)\\
        &\geq q^{m-k}(q^{j-m}-(-1)^{j-m})\:.
    \end{align*}}
    for $0\leq m\leq j-1$. From this it follows that
    \begin{align*}
        \rho_{j,k,j+k}&\geq q^{jk-\binom{j+1}{2}}\frac{\varphi^{-}_{1,j}(q)}{\varphi^{-}_{k+1,k+j}(q)}\left(\varphi^{+}_{1,j}(q)+(-1)^{k}\varphi^{+}_{1,j-1}(q)\frac{q^{j}-(-1)^{j}}{q+1}q^{-k}\right)\\
        &=q^{jk-\binom{j+1}{2}}\frac{\psi^{-}_{1,j-1}(q^2)}{\varphi^{-}_{k+1,k+j}(q)}\left(q^{j}-(-1)^{j}\right)\left(q^{j}+(-1)^{j}+(-1)^{k}\frac{q^{j}-(-1)^{j}}{q+1}q^{-k}\right)\\
        &=q^{jk-\binom{j+1}{2}+2j}\frac{\psi^{-}_{1,j-1}(q^2)}{\varphi^{-}_{k+1,k+j}(q)}\left(1-(-1)^{j}q^{-j}\right)\left(1+(-1)^{j}q^{-j}+(-1)^{k}\frac{q^{j}-(-1)^{j}}{q+1}q^{-k-j}\right)\:.
    \end{align*}
    Using Lemmas \ref{lem:psiminbounds} and \ref{lem:phiminupperbound} we find
    \begin{align*}
        \rho_{j,k,j+k}&\geq q^{jk-\binom{j}{2}+j}\frac{q^{2\binom{j}{2}}}{q^{jk+\binom{j+1}{2}}}\left(1-\frac{1}{q^{2}}-\frac{1}{q^{4}}+\frac{1}{q^{2j}}\right)\left(1-(-1)^{j}q^{-j}\right)\\
        &\qquad\qquad \frac{1+(-1)^{k}\frac{q^{j}-(-1)^{j}}{q+1}q^{-k-j}+(-1)^{j}q^{-j}}{1+(-1)^{k}\frac{q^{j}-(-1)^{j}}{q+1}q^{-k-j}-q^{-2k-j-1}}\\
        &\geq \left(1-\frac{1}{q^{2}}-\frac{1}{q^{4}}+\frac{1}{q^{2j}}\right)\left(1-(-1)^{j}q^{-j}\right)\left(1+\frac{(-1)^{j}q^{-j}}{1+(-1)^{k}\frac{q^{j}-(-1)^{j}}{q+1}q^{-k-j}-q^{-2k-j-1}}\right)\:.
    \end{align*}

    For $j$ even we then find
    \begin{align*}
        \rho_{j,k,j+k}&\geq\left(1-\frac{1}{q^{2}}-\frac{1}{q^{4}}+\frac{1}{q^{2j}}\right)\left(1-q^{-j}\right)\left(1+\frac{q^{-j}}{1+(-1)^{k}\frac{q^{j}-1}{q+1}q^{-k-j}-q^{-2k-j-1}}\right)\\
        &\geq\left(1-\frac{1}{q^{2}}-\frac{1}{q^{4}}+\frac{1}{q^{2j}}\right)\left(1-\frac{1}{q^{j}}\right)\left(1+\frac{q^{-j}}{1+\frac{q^{j}-1}{q+1}q^{-k-j}}\right)\\
        &\geq\left(1-\frac{1}{q^{2}}-\frac{1}{q^{4}}+\frac{1}{q^{2j}}\right)\left(1-\frac{1}{q^{j}}\right)\left(1+\frac{q^{-j}}{1+q^{-k-1}}\right)\\
        &\geq\left(1-\frac{1}{q^{2}}-\frac{1}{q^{4}}+\frac{1}{q^{2j}}\right)\left(1-\frac{1}{q^{j}}\right)\left(1+q^{-j}\left(1-q^{-k-1}\right)\right)\\
        &=1-\frac{1}{q^{2}}-\frac{1}{q^{4}}+\frac{1}{q^{2j+2}}+\frac{1}{q^{2j+4}}-\frac{1}{q^{4j}}-\frac{1}{q^{k+j+1}}+\frac{1}{q^{k+j+3}}+\frac{1}{q^{k+j+5}}-\frac{1}{q^{k+3j+1}}\\&\qquad\qquad+\frac{1}{q^{k+2j+1}}-\frac{1}{q^{k+2j+3}}-\frac{1}{q^{k+2j+5}}+\frac{1}{q^{k+4j+1}}\\
        &\geq 1-\frac{1}{q^{2}}-\frac{1}{q^{4}}+\frac{1}{q^{2j+2}}-\frac{1}{q^{k+j+1}}\:.
    \end{align*}
    For $j$ odd (and thus at least 3) we find
    \begin{align*}
        \rho_{j,k,j+k}&\geq\left(1-\frac{1}{q^{2}}-\frac{1}{q^{4}}+\frac{1}{q^{2j}}\right)\left(1+q^{-j}\right)\left(1-\frac{q^{-j}}{1+(-1)^{k}\frac{q^{j}+1}{q+1}q^{-k-j}-q^{-2k-j-1}}\right)\\
        &\geq\left(1-\frac{1}{q^{2}}-\frac{1}{q^{4}}+\frac{1}{q^{2j}}\right)\left(1+q^{-j}\right)\left(1-\frac{q^{-j}}{1-\left(\frac{q^{j}+1}{q+1}+q^{-k-1}\right)q^{-k-j}}\right)\\
        &\geq\left(1-\frac{1}{q^{2}}-\frac{1}{q^{4}}+\frac{1}{q^{2j}}\right)\left(1+q^{-j}\right)\left(1-\frac{q^{-j}}{1-q^{-k-1}}\right)\\
        &\geq\left(1-\frac{1}{q^{2}}-\frac{1}{q^{4}}+\frac{1}{q^{2j}}\right)\left(1+q^{-j}\right)\left(1-q^{-j}\left(1+q^{-k-1}+q^{-2k-1}\right)\right)\\
        &=1-\frac{1}{q^{2}}-\frac{1}{q^{4}}+\frac{1}{q^{2j+2}}+\frac{1}{q^{2j+4}}-\frac{1}{q^{4j}}-\frac{1}{q^{k+j+1}}+\frac{1}{q^{k+j+3}}+\frac{1}{q^{k+j+5}}-\frac{1}{q^{k+3j+1}}\\&\qquad-\frac{1}{q^{k+2j+1}}+\frac{1}{q^{k+2j+3}}+\frac{1}{q^{k+2j+5}}-\frac{1}{q^{k+4j+1}}-\frac{1}{q^{2k+j+1}}+\frac{1}{q^{2k+j+3}}+\frac{1}{q^{2k+j+5}}\\&\qquad-\frac{1}{q^{2k+3j+1}}-\frac{1}{q^{2k+2j+1}}+\frac{1}{q^{2k+2j+3}}+\frac{1}{q^{2k+2j+5}}-\frac{1}{q^{2k+4j+1}}\\
        &\geq1-\frac{1}{q^{2}}-\frac{1}{q^{4}}+\frac{1}{q^{2j+2}}-\frac{1}{q^{k+j+1}}\:.
    \end{align*}
    So, for both cases we find the same lower bound.
\end{proof}

\begin{remark}
    From the first equality we gave in Theorem \ref{th:hermitianrhobound1} we can derive immediately that $\rho_{1,1,2}=1-q^{-2}\left(\frac{q}{q-1}\right)$ and thus $\rho_{1,1,2}\geq1-\frac{3}{2}q^{-2}$ for all $q\geq3$, with equality for $q=3$. This shows that the bound in \cite{gim} was optimal amongst the bounds of the form $1-aq^{-2}$. For $k\geq2$ we can derive from the first equality that $\rho_{1,k,k+1}\geq1-q^{-2}\left(\frac{q^3-q^2+q}{(q-1)(q^{2}+1)}\right)$, hence $\rho_{1,k,k+1}\geq1-\frac{6}{5}q^{-2}$ for all $q\geq2$.
    \par Similarly, we can derive from the bound in Theorem \ref{th:hermitianrhobound1} that $\rho_{j,k,j+k}=1-q^{-2}-q^{-4}-q^{-5}+q^{-6}$ for $j,k\geq2$, hence $\rho_{j,k,k+j}\geq1-\frac{21}{16}q^{-2}$ for all $q\geq2$ and $j,k\geq2$.
\end{remark}

\begin{remark}
    The coefficients of $1$, $q^{-2}$ and $q^{-4}$ from the $j,k\geq2$ case in Theorem \ref{th:hermitianrhobound1} cannot be improved since some more detailed analysis shows that
    \begin{align*}
        \rho_{j,k,j+k}=
        \begin{cases}
            1-q^{-2}-q^{-4}+O\left(q^{-2j-1}\right) &\quad 2\leq j=k,\\
            1-q^{-2}-q^{-4}+O\left(q^{-2j-2}\right) &\quad 2\leq j<k,\\
            1-q^{-2}-q^{-4}+O\left(q^{-2k-2}\right) &\quad 2\leq k<j.
        \end{cases}
    \end{align*}
    %\sout{Note that in particular the coefficient of $q^{-2}$ is always 1. In \cite{gnp} it was proven that $\rho_{j,k,j+k}\geq 1-2q^{-2}$, which was improved to $\rho_{j,k,j+k}\geq 1-\frac{3}{2}q^{-2}$ in \cite{gim}.}
\end{remark}

\begin{theorem}\label{th:hermitianrhobound2}
    Let $j$, $k$ and $n$ be integers with $j,k\geq 1$ and $n\geq j+k+1$. We have that
    \begin{align*}
        \rho_{1,k,n}%&=q^{k-1}\frac{\varphi^{-}_{n-k,n-k}(q)}{\varphi^{-}_{n,n}(q)}\sum_{m=0}^{1}(-1)^{mk}\varphi^{+}_{1,1-m}(q)\gs{1}{m}{q}^{-}q^{\binom{m}{2}-mk}\\
        %&=q^{k-1}\frac{q^{n-k}-(-1)^{n-k}}{q^{n}-(-1)^{n}}\left((q-1)+(-1)^{k}q^{-k}\right)\\
        &=1-\frac{1}{q}+(-1)^{k}\frac{q^{n-k-1}-(-1)^{n}q^{k-1}(q-1)+(-1)^{n-k}q^{-1}(q-2)}{q^{n}-(-1)^{n}}\:,\\
        \rho_{j,1,n}&=1-\frac{1}{q}+(-1)^{j}\frac{q^{n-j-1}-(-1)^{n}q^{j-1}(q-1)+(-1)^{n-j}q^{-1}(q-2)}{q^{n}-(-1)^{n}}
    \end{align*}
    and
    \begin{align*}
        \rho_{j,k,j+k+1}&\geq 1-q^{-1}-q^{-3}-3q^{-4}\\
        \rho_{j,k,n}&\geq 1-q^{-1}+q^{-2}-4q^{-3}&n\geq j+k+2
    \end{align*}
    if $j,k\geq2$.
\end{theorem}
\begin{proof}
    The formula for the case $j=1$ follows immediately from Theorem \ref{th:rhohermitian}, and by symmetry also the formula for the case $k=1$. We now look at the the case $j,k\geq2$. Because of the symmetry we may assume without loss of generality that $j\leq k$. From Lemma \ref{lem:boundingthesummands} it follows that
    \begin{align*}
        \varphi^{+}_{1,j-m}(q)\gs{j}{m}{q}^{-}q^{\binom{m}{2}-mk}\geq\varphi^{+}_{1,j-m-1}(q)\gs{j}{m+1}{q}^{-}q^{\binom{m+1}{2}-(m+1)k}\:.
    \end{align*}
    for $0\leq m\leq j-1$. From this it follows that
    \begin{align*}
        \rho_{j,k,n}&\geq q^{jk-\binom{j+1}{2}}\frac{\varphi^{-}_{n-j-k+1,n-k}(q)}{\varphi^{-}_{n-j+1,n}(q)}\left(\varphi^{+}_{1,j}(q)+(-1)^{k}\varphi^{+}_{1,j-1}(q)\gs{j}{1}{q}^{-}q^{-k}\right)\\
        &=q^{jk-\binom{j+1}{2}}\frac{\varphi^{-}_{n-j-k+1,n-k}(q)\varphi^{+}_{1,j-1}(q)}{\varphi^{-}_{n-j+1,n}(q)}\left(q^{j}+(-1)^{j}+(-1)^{k}\gs{j}{1}{q}^{-}q^{-k}\right)\:.
    \end{align*}
    Now using Lemmas \ref{lem:phiminupperbound}, \ref{lem:phiminlowerbound} and \ref{lem:phispluslowerbound} we find 
    \begin{align*}
        \rho_{j,k,n}&\geq q^{jk-\binom{j+1}{2}+j(n-j-k)+\binom{j+1}{2}+\binom{j}{2}+j-(n-j)j-\binom{j+1}{2}}\frac{1+(-1)^{j}q^{-j}+(-1)^{k}\gs{j}{1}{q}^{-}q^{-j-k}}{1+(-1)^{n-j}q^{-n}\gs{j}{1}{q}^{-}-q^{-2n+j-1}}\\&\qquad\left(1+(-1)^{n-j-k}q^{j+k-n-1}-(-1)^{n-j-k}q^{j+k-n-2}-q^{j+k-n-3}-q^{2j+2k-2n-3}\right)\\&\qquad\left(1-\frac{1}{q}+\frac{1}{q^2}-\frac{2}{q^3}\right)\\
        &= \frac{1+(-1)^{j}q^{-j}+(-1)^{k}\gs{j}{1}{q}^{-}q^{-j-k}}{1+(-1)^{n-j}q^{-n}\gs{j}{1}{q}^{-}-q^{-2n+j-1}}\left(1-\frac{1}{q}+\frac{1}{q^2}-\frac{2}{q^3}\right)\\&\qquad\left(1+(-1)^{n-j-k}q^{j+k-n-1}-(-1)^{n-j-k}q^{j+k-n-2}-q^{j+k-n-3}-q^{2j+2k-2n-3}\right)\:.
    \end{align*}

    Now, we have that
    \begin{align*}
        &\frac{1}{1+(-1)^{n-j}q^{-n}\gs{j}{1}{q}^{-}-q^{-2n+j-1}}\geq1-(-1)^{n-j}\gs{j}{1}{q}^{-}q^{-n}\\
        \Leftrightarrow\qquad&1\geq1-q^{-2n+j-1}-\left(\gs{j}{1}{q}^{-}\right)^{2}q^{-2n}+(-1)^{n-j}\gs{j}{1}{q}^{-}q^{-3n+j-1}\:,
    \end{align*}
 
    which is true since
    \[
        q^{-2n+j-1}> q^{-3n+j-1}q^{j-1}\geq\gs{j}{1}{q}^{-}q^{-3n+j-1}\:.
    \]
       
    So, we get that

    \begin{align*}
        \rho_{j,k,n}&\geq \left(1+(-1)^{j}q^{-j}+(-1)^{k}\gs{j}{1}{q}^{-}q^{-j-k}\right)\left(1-(-1)^{n-j}q^{-n}\gs{j}{1}{q}^{-}\right)\\&\qquad\left(1+(-1)^{n-j-k}q^{j+k-n-1}-(-1)^{n-j-k}q^{j+k-n-2}-q^{j+k-n-3}-q^{2j+2k-2n-3}\right)\\&\qquad\left(1-\frac{1}{q}+\frac{1}{q^2}-\frac{2}{q^3}\right)\\
        &\geq\left(1+(-1)^{j}q^{-j}+(-1)^{k}\gs{j}{1}{q}^{-}q^{-j-k}\right)\left(1-q^{-j-k-1}\gs{j}{1}{q}^{-}\right)\\&\qquad\left(1-q^{-1}+q^{-2}-2q^{-3}+(-1)^{n-j-k}q^{j+k-n-1}\left(1-2q^{-1}+2q^{-2}-3q^{-3}+2q^{-4}\right)\right.\\&\qquad\quad\left.-q^{j+k-n-3}\left(1-q^{-1}+q^{-2}-2q^{-3}\right)-q^{2j+2k-2n-3}\left(1-q^{-1}+q^{-2}-2q^{-3}\right)\right)\\
        &=\left(1+(-1)^{j}q^{-j}+(-1)^{k}\gs{j}{1}{q}^{-}q^{-j-k-1}\left(q-(-1)^{k}-(-1)^{j-k}q^{-j}-\gs{j}{1}{q}^{-}q^{-j-k}\right)\right)\\&\qquad\left(1-q^{-1}+q^{-2}-2q^{-3}+(-1)^{n-j-k}q^{j+k-n-1}\left(1-2q^{-1}+2q^{-2}-3q^{-3}+2q^{-4}\right)\right.\\&\qquad\quad\left.-q^{j+k-n-3}\left(1-q^{-1}+q^{-2}-2q^{-3}\right)-q^{2j+2k-2n-3}\left(1-q^{-1}+q^{-2}-2q^{-3}\right)\right)\:.
    \end{align*}
    Note that $q-(-1)^{k}-(-1)^{j-k}q^{-j}-\gs{j}{1}{q}^{-}q^{-j-k}>0$ for all $j,k\geq2$. So, for a fixed $j$ the term $(-1)^{k}\gs{j}{1}{q}^{-}q^{-j-k-1}\left(q-(-1)^{k}-(-1)^{j-k}q^{-j}-\gs{j}{1}{q}^{-}q^{-j-k}\right)$ reaches its minimum for an odd value of $k$. Since also $q-(-1)^{k}-(-1)^{j-k}q^{-j}-\gs{j}{1}{q}^{-}q^{-j-k}<q^{2}$ for $j,k\geq3$, the minimum is reached for $k=3$. We find that
    \begin{align*}
        \rho_{j,k,n}&\geq\left(1+(-1)^{j}q^{-j}-\gs{j}{1}{q}^{-}q^{-j-4}\left(q+1+(-1)^{j}q^{-j}-\gs{j}{1}{q}^{-}q^{-j-3}\right)\right)\\&\qquad\left(1-q^{-1}+q^{-2}-2q^{-3}+(-1)^{n-j-k}q^{j+k-n-1}\left(1-2q^{-1}+2q^{-2}-3q^{-3}+2q^{-4}\right)\right.\\&\qquad\quad\left.-q^{j+k-n-3}\left(1-q^{-1}+q^{-2}-2q^{-3}\right)-q^{2j+2k-2n-3}\left(1-q^{-1}+q^{-2}-2q^{-3}\right)\right)\\
        &\geq\left(1+(-1)^{j}q^{-j}\left(1-\gs{j}{1}{q}^{-}q^{-j-4}\right)-q^{-4}-q^{-5}+q^{-11}(q-1)^{2}\right)\\&\qquad\left(1-q^{-1}+q^{-2}-2q^{-3}+(-1)^{n-j-k}q^{j+k-n-1}\left(1-2q^{-1}+2q^{-2}-3q^{-3}+2q^{-4}\right)\right.\\&\qquad\quad\left.-q^{j+k-n-3}\left(1-q^{-1}+q^{-2}-2q^{-3}\right)-q^{2j+2k-2n-3}\left(1-q^{-1}+q^{-2}-2q^{-3}\right)\right)\\
        &\geq\left(1-q^{-3}\left(1-(q^{2}-q+1)q^{-7}\right)-q^{-4}-q^{-5}+q^{-11}(q-1)^{2}\right)\\&\qquad\left(1-q^{-1}+q^{-2}-2q^{-3}+(-1)^{n-j-k}q^{j+k-n-1}\left(1-2q^{-1}+2q^{-2}-3q^{-3}+2q^{-4}\right)\right.\\&\qquad\quad\left.-q^{j+k-n-3}\left(1-q^{-1}+q^{-2}-2q^{-3}\right)-q^{2j+2k-2n-3}\left(1-q^{-1}+q^{-2}-2q^{-3}\right)\right)\\
        &=\left(1-q^{-3}-q^{-4}-q^{-5}+q^{-8}-q^{-10}+q^{-11}\right)\\&\qquad\left(1-q^{-1}+q^{-2}-2q^{-3}+(-1)^{n-j-k}q^{j+k-n-1}\left(1-2q^{-1}+2q^{-2}-3q^{-3}+2q^{-4}\right)\right.\\&\qquad\quad\left.-q^{j+k-n-3}\left(1-q^{-1}+q^{-2}-2q^{-3}\right)-q^{2j+2k-2n-3}\left(1-q^{-1}+q^{-2}-2q^{-3}\right)\right)\:.
        %&=\left(1+(-1)^{j}q^{-j}-(-1)^{n-j}q^{-n-j}\frac{q^{2j}-1}{q+1}\right.\\&\qquad\quad\left.+(-1)^{k}\gs{j}{1}{q}^{-}q^{-n-j-k}\left(q^{n}-(-1)^{n-j}\gs{j}{1}{q}^{-}\right)\right)\\&\qquad\left(1-q^{-1}+q^{-2}-2q^{-3}+(-1)^{n-j-k}q^{j+k-n-1}\left(1-2q^{-1}+2q^{-2}-3q^{-3}+2q^{-4}\right)\right.\\&\qquad\quad\left.-q^{j+k-n-3}\left(1-q^{-1}+q^{-2}-2q^{-3}\right)-q^{2j+2k-2n-3}\left(1-q^{-1}+q^{-2}-2q^{-3}\right)\right)\\
        %&\geq\left(1-q^{-j}-\frac{q^{j}+1}{q+1}q^{-j-k}-q^{-n-j}\frac{q^{2j}-1}{q+1}+\left(\frac{q^{j}+1}{q+1}\right)^{2}q^{-n-j-k}\right)\\&\qquad\left(1-q^{-1}+q^{-2}-2q^{-3}-q^{j+k-n-1}\left(1-2q^{-1}+3q^{-2}-4q^{-3}+3q^{-4}-2q^{-5}\right)\right.\\&\qquad\quad\left.-q^{2j+2k-2n-3}\left(1-q^{-1}+q^{-2}-2q^{-3}\right)\right)\\
        %&\geq\left(1-q^{-j}-q^{-k-1}-q^{-n+j-1}+\left(q^{2j-2}-2q^{2j-3}\right)q^{-n-j-k}\right)\\&\qquad\left(1-q^{-1}+q^{-2}-2q^{-3}-q^{j+k-n-1}\left(1-2q^{-1}+3q^{-2}-4q^{-3}+3q^{-4}-2q^{-5}\right)\right.\\&\qquad\quad\left.-q^{2j+2k-2n-3}\left(1-q^{-1}+q^{-2}-2q^{-3}\right)\right)\:,
    \end{align*}
   
    Now we distinguish between two cases. If $n=j+k+1$ we find
    \begin{align*}
        \rho_{j,k,n}&\geq\left(1-q^{-3}-q^{-4}-q^{-5}+q^{-8}-q^{-10}+q^{-11}\right)\\&\qquad\left(1-q^{-1}-3q^{-4}+3q^{-5}-2q^{-6}+q^{-7}+2q^{-8}\right)\\
        &=1-q^{-1}-q^{-3}-3q^{-4}+3q^{-5}-q^{-6}+4q^{-7}+3q^{-8}+q^{-9}-3q^{-10}+q^{-11}-7q^{-12}\\&\qquad+q^{-13}+q^{-14}-5q^{-15}+7q^{-16}-3q^{-17}-q^{-18}+2q^{-19}\\
        &\geq1-q^{-1}-q^{-3}-3q^{-4}\:.
    \end{align*}
     Note that $1-2q^{-1}+2q^{-2}-3q^{-3}+2q^{-4}\geq0$ and $1-q^{-1}+q^{-2}-2q^{-3}\geq0$ for $q\geq2$, so for $n\geq j+k+2$ we find
    \begin{align*}
        \rho_{j,k,n}&\geq\left(1-q^{-3}-q^{-4}-q^{-5}+q^{-8}-q^{-10}+q^{-11}\right)\\&\qquad\left(1-q^{-1}+q^{-2}-3q^{-3}+2q^{-4}-3q^{-5}+4q^{-6}-4q^{-7}+3q^{-8}-q^{-9}+2q^{-10}\right)\\
        &=1-q^{-1}+q^{-2}-4q^{-3}+2q^{-4}-4q^{-5}+7q^{-6}-4q^{-7}+8q^{-8}-5q^{-9}+5q^{-10}-4q^{-11}\\&\qquad+2q^{-12}-3q^{-13}-2q^{-14}-q^{-15}-4q^{-16}+7q^{-17}-5q^{-18}+4q^{-19}-3q^{-20}+2q^{-21}\\
        &\geq1-q^{-1}+q^{-2}-4q^{-3}\:.\qedhere
    \end{align*}
\end{proof}

\begin{remark} We can derive the following from Theorem \ref{th:hermitianrhobound2}: 
\begin{itemize}
\item for $j,k\geq 2$ and $n\geq j+k+2$: $\rho_{j,k,n}\geq 1-\frac{3}{2}q^{-1}$,
\item for $j,k\geq 2$: $\rho_{j,k,j+k+1}\geq 1-\frac{13}{8}q^{-1}$,
%\item For $k,n$ even, for $k$ even, $n\geq 2k+1$, and for $k>n-k-1$ odd, $n$ even, we have $\rho_{1,k,n}\geq 1-\frac{1}{q}$.
%\item For $k$ even, $n$ odd, $n<2k+1$, $\rho_{1,k,n}\geq 1-\frac{5}{4}\frac{1}{q}$
%\item in all other cases, $\rho_{1,k,n}\geq 1-\frac{3}{2}\frac{1}{q}$
\item for $k\geq 1$ and $n\geq k+2$: $\rho_{1,k,n}\geq 1-\frac{5}{3}q^{-1}$. 
\end{itemize}
\end{remark}

\begin{remark}
    The coefficients of $1$, $q^{-1}$ and $q^{-2}$ from the $j,k\geq2$ case in Theorem \ref{th:hermitianrhobound2} can not be improved since some more detailed analysis shows that
    \begin{align*}
        \rho_{j,k,n}=
        \begin{cases}
            1-q^{-1}+O\left(q^{-3}\right) &\quad n=j+k+1\text{ and }j,k\geq2,\\
            1-q^{-1}+q^{-2}+O\left(q^{-3}\right) &\quad n\geq j+k+2\text{ and }j,k\geq2.
        \end{cases}
    \end{align*}
    %\sout{Note that in particular the coefficient of $q^{-1}$ is always 1. In \cite{gnp2} it was proven that $\rho_{j,k,n}\geq 1-\frac{43}{25}q^{-1}$ if $n\geq j+k+1$.}
\end{remark}

\section{The symplectic case}
\subsection{Proof of the main result (symplectic)}\label{sec:symplectic}

We first introduce some notation, which will be completely analogous as those for the hermitian case. Note that we are working in an even-dimensional vector space here.

\begin{definition}
    Given a non-degenerate symplectic form $f$ on $\F_{q}^{2n}$, %the $j$-spaces $\pi$ in $\F_{q}^{2n}$ such that the restriction on $\pi$ has a $i$-dimensional radical are called $i$-singular. A 0-singular space is also called non-singular. W
    we define $\alpha_{i,j,2n}$ as the number of $i$-singular $j$-spaces with respect to $f$.
\end{definition}

The symplectic group $\PGSp(2n,q)$, i.e.~the group of all linear maps preserving the non-degenerate symplectic form up to scalar multiple, acts transitively on the $i$-singular $2j$-spaces (see \cite[Theorem 3.7]{wanbook}), so the following is well-defined.

\begin{definition}
	Consider a non-degenerate symplectic form $f$ on $\F_{q}^{2n}$.
	\begin{itemize}
		%\item $\alpha_{i,j,2n}$ is the number of $j$-spaces $\pi$ in $\F_{q}^{2n}$ such that the restriction on $\pi$ has a $i$-dimensional radical. In other words, this is the number of $i$-singular $j$-spaces on $\mathcal{W}(2n-1,q)$.
		\item Given an $i$-singular $j$-space $\pi$ in $\F_{q}^{2n}$, $\beta_{i,j,2n,2k}$ is the number of non-singular $2k$-spaces $\sigma\supseteq\pi$ in $\F_{q}^{2n}$.
		\par For the case of co-dimension 2-spaces, $2k=2n-2$, we use the notation $\beta_{i,j,2n}=\beta_{i,j,2n,2n-2}$.
		\item Given an $i$-singular $j$-space $\pi$ in $\F_{q}^{2n}$, $\gamma_{i,j,2n,2k}$ is the number of non-singular $2k$-spaces $\sigma$ in $\F_{q}^{2n}$ such that $\sigma\cap\pi$ is trivial, and $\langle\pi,\sigma\rangle$ is a non-singular $(2k+j)$-space.
		\par For the case $2k=2n-j$, where the two spaces span $\F_{q}^{2n}$ we use the notation $\gamma_{i,j,2n}=\gamma_{i,j,2n,2n-j}$.
	\end{itemize}
\end{definition}

Throughout this section the underlying field $\F_q$ is fixed. For this reason we have omitted the $q$ in the notation $\alpha_{i,j,2n}$, $\beta_{i,j,2n,2k}$, $\beta_{i,j,2n}$, $\gamma_{i,j,2n,2k}$ and $\gamma_{i,j,2n}$. As for the hermitian case, the expression for $\alpha_{i,j,2n}$ is known in the literature. Since an $i$-singular $j$-space necessarily has $j-i$ even, we will be able to restrict ourselves to values of $\alpha_{i,j,2n}$ of the form $\alpha_{j-2l,j,2n}$.

\begin{lemma}[{\cite[Theorem 3.18]{wanbook}}]\label{lem:alphasymplectic}
	If $i\equiv j+1\pmod{2}$, then $\alpha_{i,j,2n}=0$. 
	For $\max\{0,j-n\}\leq \ell\leq \frac{j}{2}$ and $j\leq 2n$ we have that
	\begin{align*}
		\alpha_{j-2\ell,j,2n}(q)%&=q^{2\ell(n-j+\ell)}\frac{\prod_{m=n+\ell-j+1}^{n}\left(q^{2m}-1\right)}{\prod_{m=1}^{\ell}\left(q^{2m}-1\right)\prod_{m=1}^{j-2\ell}\left(q^{m}-1\right)}\\
        &=q^{2\ell(n-j+\ell)}\frac{\psi^{-}_{n+\ell-j+1,n}(q^{2})}{\psi^{-}_{1,\ell}(q^{2})\psi^{-}_{1,j-2\ell}(q)}%\\&
        =q^{2\ell(n-j+\ell)}\gs{n}{\ell}{q^2}\gs{n-\ell}{j-2\ell}{q}\psi^{+}_{n+\ell-j+1,n-\ell}(q)\:.
	\end{align*}
\end{lemma}

We are now ready to derive a formula for $\beta_{i,j,2n,2k}$.
\begin{lemma}\label{lem:betasymplectic}
	If $i\equiv j+1\pmod{2}$, then $\beta_{i,j,2n,2k}$ is undefined. 
	For $\max\{0,j-k\}\leq \ell\leq \frac{j}{2}$ and $j\leq 2k\leq 2n-2$ we have that
	\[
		\beta_{j-2\ell,j,2n,2k}=q^{2(k-\ell)(n-k)}\gs{n-j+\ell}{k-j+\ell}{q^2}\:.
	\]
\end{lemma}
\begin{proof}
	We count the tuples $(\pi,\sigma)$ with $\pi$ a $(j-2\ell)$-singular $j$-space and $\sigma\supseteq\pi$ a non-singular $2k$-space, in two ways. We find that
	\[
		\alpha_{j-2\ell,j,2n}\beta_{j-2\ell,j,2n,2k}=\alpha_{0,2k,2n}\alpha_{j-2\ell,j,2k}\:.
	\]
	Hence, using Lemma \ref{lem:alphasymplectic} we find
	%\begin{align*}
		%\beta_{j-2\ell,j,2n,2k}&=\frac{\alpha_{0,2k,2n}\:\alpha_{j-2\ell,j,2k}}{\alpha_{j-2\ell,j,2n}}\\
		%&=\frac{q^{2k(n-k)}\gs{n}{k}{q^2}\:q^{2\ell(k-j+\ell)}\gs{k}{\ell}{q^2}\gs{k-\ell}{j-2\ell}{q}\psi^{+}_{k+\ell-j+1,k-\ell}(q)}{q^{2\ell(n-j+\ell)}\gs{n}{\ell}{q^2}\gs{n-\ell}{j-2\ell}{q}\psi^{+}_{n+\ell-j+1,n-\ell}(q)}\\
		%&=q^{2(k-\ell)(n-k)}\gs{n-\ell}{k-\ell}{q^2}\frac{\psi^{-}_{k-j+\ell+1,k-\ell}(q)\psi^{+}_{k+\ell-j+1,k-\ell}(q)}{\psi^{-}_{n-j+\ell+1,n-\ell}(q)\psi^{+}_{n+\ell-j+1,n-\ell}(q)}\\
		%&=q^{2(k-\ell)(n-k)}\gs{n-\ell}{k-\ell}{q^2}\frac{\psi^{-}_{k-j+\ell+1,k-\ell}(q^2)}{\psi^{-}_{n+\ell-j+1,n-\ell}(q^2)}\\
		%&=q^{2(k-\ell)(n-k)}\gs{n-j+\ell}{k-j+\ell}{q^2}\:.%\qedhere
	%\end{align*}{\color{orange}
    \begin{align*}
		\beta_{j-2\ell,j,2n,2k}&=\frac{\alpha_{0,2k,2n}\:\alpha_{j-2\ell,j,2k}}{\alpha_{j-2\ell,j,2n}}\\
		&=\frac{q^{2k(n-k)}\psi^{-}_{n-k+1,n}(q^{2})}{\psi^{-}_{1,k}(q^{2})}\frac{q^{2\ell(k-j+\ell)}\psi^{-}_{k+\ell-j+1,k}(q^{2})}{\psi^{-}_{1,\ell}(q^{2})\psi^{-}_{1,j-2\ell}(q)}\frac{\psi^{-}_{1,\ell}(q^{2})\psi^{-}_{1,j-2\ell}(q)}{q^{2\ell(n-j+\ell)}\psi^{-}_{n+\ell-j+1,n}(q^{2})}\\
		&=q^{2(k-\ell)(n-k)}\frac{\psi^{-}_{n-k+1,n-j+\ell}(q^{2})}{\psi^{-}_{1,k-j+\ell}(q^{2})}\\
		&=q^{2(k-\ell)(n-k)}\gs{n-j+\ell}{k-j+\ell}{q^2}\:.\qedhere
	\end{align*}
\end{proof}

\begin{corollary}\label{cor:betasymplectic}
	If $i\equiv j+1\pmod{2}$, then $\beta_{i,j,2n}$ is undefined. 
	For $\max\{0,j-n+1\}\leq \ell\leq \frac{j}{2}$ and $j\leq 2n-2$ we have that
	\[
		\beta_{j-2\ell,j,2n}(q)=q^{2(n-\ell-1)}\gs{n-j+\ell}{1}{q^2}\:.
	\]
\end{corollary}

The proof of the main theorem uses roughly the same idea as in the hermitian case: we use a double counting argument in order to find a recursive expression for $\gamma_{2i,2j,2n}$ \eqref{eq:symplecticpreinduction}. However, in this case, we cannot use hyperplanes for the counting argument since no hyperplane of $\mathbb{F}^{2n}$ is non-singular, so we use co-dimension $2$-spaces. This complicates matters slightly; in particular, the coefficients in the recursive expression become more complicated.

\begin{theorem}\label{th:gammasymplectic}
	If $i$ or $j$ is odd, then $\gamma_{i,j,2n}=0$.
	For $0\leq i\leq \min\{j,n-j\}$ and $j\leq n-1$ we have that
	\[
		\gamma_{2i,2j,2n}=q^{j(4n-5j)}\chi_{1,i}(q)\sum_{m=0}^{j-i}\chi_{i+1,j-m}(q)\gs{j-i}{m}{q^{2}}q^{m(2j+2i-2n+m-1)}\:.
	\]
\end{theorem}
\begin{proof}
	We prove this theorem using induction on $j$. One can see directly that $\gamma_{0,0,2n}=1$ for all $n\geq1$. Now, consider a symplectic form on $\F^{2n}_{q}$ and a fixed $2i$-singular $2j$-space $\pi$ with respect to it, with $j\geq1$. We denote the singular $2i$-dimensional subspace of $\pi$ by $\overline{\pi}$. We count the tuples $(\sigma,\tau)$ with $\sigma$ a non-singular $(2n-2)$-space, $\tau\subseteq\sigma$ a non-singular $(2n-2j)$-space disjoint from $\pi$. Note that $\langle\pi,\tau\rangle=\F^{2n}_{q}$. On the one hand there are $\gamma_{2i,2j,2n}$ choices for $\tau$, and for each of them $\beta_{0,2n-2j,2n}$ corresponding tuples. Now, we consider the non-singular $(2n-2)$-spaces. Each non-singular $(2n-2)$-space that contains a $(2n-2j)$-space disjoint from $\pi$ meets $\pi$ in a $(2j-2)$-space, so we look at the $(2j-2)$-spaces of $\pi$. We see that there are four possibilities.
	\begin{itemize}
		\item We first look at the $(2j-2)$-spaces through $\overline{\pi}$ that are $(2i+2)$-singular. The number of such $(2j-2)$-spaces corresponds to the number of 2-singular $(2j-2i-2)$-spaces with respect to a non-degenerate symplectic form on a $(2j-2i)$ space, which is $\alpha_{2,2j-2i-2,2j-2i}$. Any such $(2j-2)$-space is contained in $q+1$ hyperplanes ($(2j-1)$-spaces) of $\pi$, all of which are $(2i+1)$-singular. Hence, through such a $(2j-2)$-space there are $\beta_{2i+2,2j-2,2n}-(q+1)(\beta_{2i+1,2j-1,2n}-\beta_{2i,2j,2n})-\beta_{2i,2j,2n}$ non-singular $(2n-2)$-spaces not containing $\pi$. For each of these hyperplanes we have $\gamma_{2i+2,2j-2,2n-2}$ tuples.
		\item Secondly, we look at the $(2j-2)$-spaces through $\overline{\pi}$ that are $2i$-singular. The number of such $(2j-2)$-spaces corresponds to the number of non-singular $(2j-2i-2)$-spaces with respect to a non-singular symplectic form on a $(2j-2i)$ space, which is $\alpha_{0,2j-2i-2,2j-2i}$. Any such $(2j-2)$-space is contained in $q+1$ hyperplanes ($(2j-1)$-spaces) of $\pi$, all of which are $(2i+1)$-singular. Hence, through such a $(2j-2)$-space there are $\beta_{2i,2j-2,2n}-(q+1)(\beta_{2i+1,2j-1,2n}-\beta_{2i,2j,2n})-\beta_{2i,2j,2n}$ non-singular $(2n-2)$-spaces not containing $\pi$. For each of these hyperplanes we have $\gamma_{2i,2j-2,2n-2}$ tuples.
		\item We now look at the $(2j-2)$-spaces that meet $\overline{\pi}$ in precisely a $(2i-1)$-space. Necessarily they are $2i$-singular. The number of such $(2j-2)$-spaces is $\gs{2i}{2i-1}{q}\left(q^{2j-2i-1}\gs{2j-2i}{2j-2i-1}{q}\right)$, where the second factor follows from Lemma \ref{lem:segre}. Any such $(2j-2)$-space is contained in $q+1$ hyperplanes ($(2j-1)$-spaces) of $\pi$, one of which contains $\overline{\pi}$ and is thus $(2i+1)$-singular, while all the others are $(2i-1)$-singular. Hence, through such a $(2j-2)$-space there are $\beta_{2i,2j-2,2n}-(\beta_{2i+1,2j-1,2n}-\beta_{2i,2j,2n}-q(\beta_{2i-1,2j-1,2n}-\beta_{2i-1,2j,2n})-\beta_{2i,2j,2n}$ non-singular $(2n-2)$-spaces not containing $\pi$. For each of these hyperplanes we have $\gamma_{2i,2j-2,2n-2}$ tuples.
		\item Finally, we look at the $(2j-2)$-spaces that meet $\overline{\pi}$ in precisely a $(2i-2)$-space. Necessarily they are $(2i-2)$-singular. The number of such $(2j-2)$-spaces is $\gs{2i}{2i-2}{q}q^{2(2j-i)}$, where the second factor follows from Lemma \ref{lem:segre}. Any such $(2j-2)$-space is contained in $q+1$ hyperplanes ($(2j-1)$-spaces) of $\pi$, all of which are $(2i-1)$-singular. Hence, through such a $(j-2)$-space there are $\beta_{2i-2,2j-2,2n}-(q+1)(\beta_{2i-1,2j-1,2n}-\beta_{2i,2j,2n})-\beta_{2i,2j,2n}$ non-singular $(2n-2)$-spaces not containing $\pi$. For each of these hyperplanes we have $\gamma_{2i-2,2j-2,2n-2}$ tuples.
	\end{itemize}
	We find the following result.
	\begin{align}\label{eq:symplecticpreinduction}
		&\gamma_{2i,2j,2n}\beta_{0,2n-2j,2n}\nonumber\\&=\alpha_{2,2j-2i-2,2j-2i}\left(\beta_{2i+2,2j-2,2n}-(q+1)\beta_{2i+1,2j-1,2n}+q\beta_{2i,2j,2n}\right)\gamma_{2i+2,2j-2,2n-2}\nonumber\\
		&\qquad+\alpha_{0,2j-2i-2,2j-2i}\left(\beta_{2i,2j-2,2n}-(q+1)\beta_{2i+1,2j-1,2n}+q\beta_{2i,2j,2n}\right)\gamma_{2i,2j-2,2n-2}\nonumber\\
		&\qquad+\gs{2i}{1}{q}q^{2j-2i-1}\gs{2j-2i}{1}{q}\gamma_{2i,2j-2,2n-2}\nonumber\\&\qquad\qquad\left(\beta_{2i,2j-2,2n}-\beta_{2i+1,2j-1,2n}-q\beta_{2i-1,2j-1,2n}+q\beta_{2i,2j,2n}\right)\nonumber\\
		&\qquad+\gs{2i}{2}{q}q^{2(2j-2i)}\left(\beta_{2i-2,2j-2,2n}-(q+1)\beta_{2i-1,2j-1,2n}+q\beta_{2i,2j,2n}\right)\gamma_{2i-2,2j-2,2n-2}
	\end{align}
	Note that this equality is also valid if $i=j$ or $i=j-1$, or $i=0$. Then, only one, three or two of the cases appear, respectively. But the cases that do not appear, have a factor 0 in Equation \ref{eq:symplecticpreinduction}. Using Corollary \ref{cor:betasymplectic} we know that
	\begin{align*}
		&\beta_{2i+2,2j-2,2n}-(q+1)\beta_{2i+1,2j-1,2n}+q\beta_{2i,2j,2n}\\ 
		&=q^{2(n-j+i+1)}\gs{n-j-i}{1}{q^2}-(q+1)q^{2(n-j+i)}\gs{n-j-i}{1}{q^2}+q\:q^{2(n-j+i-1)}\gs{n-j-i}{1}{q^2}\\
		&=q^{2(n-j+i)-1}(q-1)(q^{2(n-j-i)}-1)\\[1em]
		&\beta_{2i,2j-2,2n}-(q+1)\beta_{2i+1,2j-1,2n}+q\beta_{2i,2j,2n}\\
		&=q^{2(n-j+i)}\gs{n-j-i+1}{1}{q^2}-(q+1)q^{2(n-j+i)}\gs{n-j-i}{1}{q^2}+q\:q^{2(n-j+i-1)}\gs{n-j-i}{1}{q^2}\\
		&=q^{2(n-j+i)-1}\left(q^{2(n-j-i)}(q-1)+1\right)\\[1em]
		&\beta_{2i,2j-2,2n}-\beta_{2i+1,2j-1,2n}-q\beta_{2i-1,2j-1,2n}+q\beta_{2i,2j,2n}\\
		&=q^{2(n-j+i)}\gs{n-j-i+1}{1}{q^2}-q^{2(n-j+i)}\gs{n-j-i}{1}{q^2}-q\:q^{2(n-j+i-1)}\gs{n-j-i+1}{1}{q^2}\\&\qquad+q\:q^{2(n-j+i-1)}\gs{n-j-i}{1}{q^2}\\
		&=q^{4(n-j)-1}(q-1)\\[1em]
		&\beta_{2i-2,2j-2,2n}-(q+1)\beta_{2i-1,2j-1,2n}+q\beta_{2i,2j,2n}\\
		&=q^{2(n-j+i-1)}\gs{n-j-i+2}{1}{q^2}-(q+1)q^{2(n-j+i-1)}\gs{n-j-i+1}{1}{q^2}\\&\qquad+q\:q^{2(n-j+i-1)}\gs{n-j-i}{1}{q^2}\\
		&=q^{4(n-j)-1}(q-1)%\\[1em]
	\end{align*}
	Together with Lemma \ref{lem:alphasymplectic} this allows us to rewrite Equation \ref{eq:symplecticpreinduction} as follows 
	\begin{align}\label{eq:symplecticinduction}
		\gamma_{2i,2j,2n}\beta_{0,2n-2j,2n}&=\gs{j-i}{2}{q^2}\psi^{+}_{1,2}(q)q^{2(n-j+i)-1}(q-1)(q^{2(n-j-i)}-1)\gamma_{2i+2,2j-2,2n-2}\nonumber\\
		&\qquad+q^{2(j-i-1)}\gs{j-i}{1}{q^2}q^{2(n-j+i)-1}\left(q^{2(n-j-i)}(q-1)+1\right)\gamma_{2i,2j-2,2n-2}\nonumber\\
		&\qquad+\gs{2i}{1}{q}q^{2(j-i)-1}\gs{2j-2i}{1}{q}q^{4(n-j)-1}(q-1)\gamma_{2i,2j-2,2n-2}\nonumber\\
		&\qquad+\gs{2i}{2}{q}q^{4(j-i)}q^{4(n-j)-1}(q-1)\gamma_{2i-2,2j-2,2n-2}\nonumber\\
		&=\frac{\left(q^{2(j-i)}-1\right)\left(q^{2(j-i-1)}-1\right)}{q^{2}-1}q^{2(n-j+i)-1}(q^{2(n-j-i)}-1)\gamma_{2i+2,2j-2,2n-2}\nonumber\\
		&\qquad+q^{2n-3}\frac{q^{2(j-i)}-1}{q^2-1}\left(q^{2(n-j-i)}\left(q^{2(i+1)}+q^{2i+1}-q^2-1\right)+1\right)\gamma_{2i,2j-2,2n-2}\nonumber\\
		&\qquad+\frac{\left(q^{2i}-1\right)\left(q^{2i-1}-1\right)}{q^{2}-1}q^{4(n-i)-1}\gamma_{2i-2,2j-2,2n-2}
	\end{align}
	Using the induction hypothesis we can now rewrite Equation \ref{eq:symplecticinduction} as follows:
	\begin{align*}
		&\gamma_{2i,2j,2n}\beta_{0,2n-2j,2n}\\
		&=\frac{\left(q^{2(j-i)}-1\right)\left(q^{2(j-i-1)}-1\right)}{q^{2}-1}q^{2(n-j+i)-1}(q^{2(n-j-i)}-1)\\&\qquad\qquad q^{(j-1)(4n-5j+1)}\chi_{1,i+1}(q)\sum_{m=0}^{j-i-2}\chi_{i+2,j-m-1}(q)\gs{j-i-2}{m}{q^{2}}q^{m(2j+2i-2n+m+1)}\\%\gamma_{2i+2,2j-2,2n-2} i->i+1,j->j-1,n->n-1
		&\qquad+q^{2n-3}\frac{q^{2(j-i)}-1}{q^2-1}\left(q^{2(n-j-i)}\left(q^{2(i+1)}+q^{2i+1}-q^2-1\right)+1\right)\\&\qquad\qquad q^{(j-1)(4n-5j+1)}\chi_{1,i}(q)\sum_{m=0}^{j-i-1}\chi_{i+1,j-m-1}(q)\gs{j-i-1}{m}{q^{2}}q^{m(2j+2i-2n+m-1)}\\%\gamma_{2i,2j-2,2n-2} i->i, j->j-1,n->n-1
		&\qquad+\frac{\left(q^{2i}-1\right)\left(q^{2i-1}-1\right)}{q^{2}-1}q^{4(n-i)-1}\\&\qquad\qquad q^{(j-1)(4n-5j+1)}\chi_{1,i-1}(q)\sum_{m=0}^{j-i}\chi_{i,j-m-1}(q)\gs{j-i}{m}{q^{2}}q^{m(2j+2i-2n+m-3)}\\%\gamma_{2i-2,2j-2,2n-2} i->i-1,j->j-1,n->n-1
		&=q^{j(4n-5j)}\chi_{1,i}(q)\frac{q^{2j-2}}{q^2-1}\left(q^{2(j+i-n)}(q^{2(n-j-i)}-1)\left(q^{2(j-i)}-1\right)\vphantom{\sum_{m=0}^{j-i}\gs{j-i}{m}{q^{2}}}\right.\\&\qquad\qquad \sum_{m=0}^{j-i-2}\chi_{i+1,j-m-1}(q)\gs{j-i-1}{m}{q^{2}}\left(q^{2(j-i-m-1)}-1\right)q^{m(2j+2i-2n+m+1)}\\
		&\qquad+q^{2(2j-n-1)}\left(q^{2(j-i)}-1\right)\left(q^{2(n-j-i)}\left(q^{2(i+1)}+q^{2i+1}-q^2-1\right)+1\right)\\&\qquad\qquad \sum_{m=0}^{j-i-1}\chi_{i+1,j-m-1}(q)\gs{j-i-1}{m}{q^{2}}q^{m(2j+2i-2n+m-1)}\\
		&\qquad+\left(q^{2i}-1\right)q^{4(j-i)}\left. \sum_{m=0}^{j-i}\chi_{i,j-m-1}(q)\gs{j-i}{m}{q^{2}}q^{m(2j+2i-2n+m-3)}\right)\\
		&=q^{j(4n-5j)}\chi_{1,i}(q)\frac{q^{2j-2}}{q^2-1}\left(q^{2(2j-n-1)}(q^{2(n-j-i)}-1)\left(q^{2(j-i)}-1\right)\vphantom{\sum_{m=0}^{j-i}\gs{j-i}{m}{q^{2}}}\right.\\&\qquad\qquad \sum_{m=0}^{j-i-2}\chi_{i+1,j-m-1}(q)\gs{j-i-1}{m}{q^{2}}q^{-2m}q^{m(2j+2i-2n+m+1)}\\
		&\qquad-q^{2(j+i-n)}(q^{2(n-j-i)}-1)\left(q^{2(j-i)}-1\right)%\\&\qquad\qquad 
		\sum_{m=0}^{j-i-2}\chi_{i+1,j-m-1}(q)\gs{j-i-1}{m}{q^{2}}q^{m(2j+2i-2n+m+1)}\\
		&\qquad-q^{2(2j-n-1)}\left(q^{2(j-i)}-1\right)\left(q^{2(n-j-i)}-1\right)\\&\qquad\qquad		\sum_{m=0}^{j-i-1}\chi_{i+1,j-m-1}(q)\gs{j-i-1}{m}{q^{2}}q^{m(2j+2i-2n+m-1)}\\
		&\qquad+q^{2(j-i)}\left(q^{2(j-i)}-1\right)\left(q^{2i}+q^{2i-1}-1\right)%\\&\qquad\qquad 
		\sum_{m=0}^{j-i-1}\chi_{i+1,j-m-1}(q)\gs{j-i-1}{m}{q^{2}}q^{m(2j+2i-2n+m-1)}\\
		&\qquad+\left(q^{2i}-1\right)q^{4(j-i)}\left. \sum_{m=0}^{j-i}\chi_{i,j-m-1}(q)\gs{j-i}{m}{q^{2}}q^{m(2j+2i-2n+m-3)}\right)\\
		&=q^{j(4n-5j)}\chi_{1,i}(q)\frac{q^{2j-2}}{q^2-1}\left(-q^{2(j+i-n)}(q^{2(n-j-i)}-1)\left(q^{2(j-i)}-1\right)\vphantom{\sum_{m=0}^{j-i}\gs{j-i}{m}{q^{2}}}\right.\\&\qquad\qquad\sum_{m=0}^{j-i-2}\chi_{i+1,j-m-1}(q)\gs{j-i-1}{m}{q^{2}}q^{m(2j+2i-2n+m+1)}\\
		&\qquad-q^{2(2j-n-1)}\left(q^{2(j-i)}-1\right)\left(q^{2(n-j-i)}-1\right)q^{(j-i-1)(3j+i-2n-2)}\\
		&\qquad+q^{2(j-i)}\left(q^{2(j-i)}-1\right)q^{2i-1}\sum_{m=0}^{j-i-1}\chi_{i+1,j-m-1}(q)\gs{j-i-1}{m}{q^{2}}q^{m(2j+2i-2n+m-1)}\\
		&\qquad+q^{2(j-i)}\left(q^{2(j-i)}-1\right)\left(q^{2i}-1\right)\sum_{m=0}^{j-i-1}\chi_{i+1,j-m-1}(q)\gs{j-i-1}{m}{q^{2}}q^{m(2j+2i-2n+m-1)}\\
        &\qquad+\left(q^{2i}-1\right)q^{4(j-i)}q^{(j-i)(3j+i-2n-3)}\\
		&\qquad+\left(q^{2i}-1\right)q^{4(j-i)}\left. \sum_{m=0}^{j-i-1}\chi_{i+1,j-m-1}(q)\left(q^{2i-1}-1\right)\gs{j-i}{m}{q^{2}}q^{m(2j+2i-2n+m-3)}\right)\\
		&=q^{j(4n-5j)}\chi_{1,i}(q)\frac{q^{2j-2}}{q^2-1}\left(-\left(q^{2(j-i)}-1\right)\left(q^{2(j-i-1)}-q^{2(2j-n-1)}\right)q^{(j-i-1)(3j+i-2n-2)}\vphantom{\sum_{m=0}^{j-i}\gs{j-i}{m}{q^{2}}}\right.\\
        &\qquad+q^{2(j+i-n)}\left(q^{2(j-i)}-1\right)\sum_{m=0}^{j-i-2}\chi_{i+1,j-m-1}(q)\gs{j-i-1}{m}{q^{2}}q^{m(2j+2i-2n+m+1)}\\
        &\qquad-\left(q^{2(j-i)}-1\right)\sum_{m=0}^{j-i-2}\chi_{i+1,j-m-1}(q)\gs{j-i-1}{m}{q^{2}}q^{m(2j+2i-2n+m+1)}\\
		&\qquad+q^{2(j-i)}\left(q^{2(j-i)}-1\right)q^{2i-1}\sum_{m=0}^{j-i-1}\chi_{i+1,j-m-1}(q)\gs{j-i-1}{m}{q^{2}}q^{m(2j+2i-2n+m-1)}\\
		&\qquad+q^{2(j-i)}\left(q^{2i}-1\right)\sum_{m=0}^{j-i-1}\chi_{i+1,j-m-1}(q)\gs{j-i}{m}{q^{2}}\left(q^{2(j-i-m)}-1\right)q^{m(2j+2i-2n+m-1)}\\
        &\qquad+\left(q^{2i}-1\right)q^{2(j-i)}q^{(j-i)(3j+i-2n-1)}\\
		&\qquad+\left(q^{2i}-1\right)q^{2(j-i)}\left. \sum_{m=0}^{j-i-1}\chi_{i+1,j-m-1}(q)\left(q^{2j-2m-1}-q^{2j-2i-2m}\right)\gs{j-i}{m}{q^{2}}q^{m(2j+2i-2n+m-1)}\right)\\
		&=q^{j(4n-5j)}\chi_{1,i}(q)\frac{q^{2j-2}}{q^2-1}\left(-\left(q^{2(j-i)}-1\right)\left(q^{2(j-i-1)}-q^{2(2j-n-1)}\right)q^{(j-i-1)(3j+i-2n-2)}\vphantom{\sum_{m=0}^{j-i}\gs{j-i}{m}{q^{2}}}\right.\\
        &\qquad+\sum_{m=0}^{j-i-2}\chi_{i+1,j-m-1}(q)\gs{j-i}{m+1}{q^{2}}\left(q^{2m+2}-1\right)q^{(m+1)(2j+2i-2n)+m(m+1)}\\
        &\qquad-\sum_{m=0}^{j-i-1}\chi_{i+1,j-m-1}(q)\gs{j-i}{m}{q^{2}}\left(q^{2(j-i-m)}-1\right)q^{m(2j+2i-2n+m+1)}\\
        &\qquad+\left(q^{2(j-i)}-1\right)q^{(j-i-1)(3j+i-2n)}+\left(q^{2j}-q^{2(j-i)}\right)q^{(j-i)(3j+i-2n-1)}\\
		&\qquad+q^{2j-1}\sum_{m=0}^{j-i-1}\chi_{i+1,j-m-1}(q)\gs{j-i}{m}{q^{2}}\left(q^{2(j-i-m)}-1\right)q^{m(2j+2i-2n+m-1)}\\
		&\qquad+\left(q^{2i}-1\right)q^{2(j-i)}\left. \sum_{m=0}^{j-i-1}\chi_{i+1,j-m-1}(q)\left(q^{2j-2m-1}-1\right)\gs{j-i}{m}{q^{2}}q^{m(2j+2i-2n+m-1)}\right)\\
		&=q^{j(4n-5j)}\chi_{1,i}(q)\frac{q^{2j-2}}{q^2-1}\left(\left(q^{2(j-i)}-1\right)q^{2(2j-n-1)}q^{(j-i-1)(3j+i-2n-2)}\vphantom{\sum_{m=0}^{j-i}\gs{j-i}{m}{q^{2}}}\right.\\
        &\qquad+\sum_{m=1}^{j-i-1}\chi_{i+1,j-m}(q)\gs{j-i}{m}{q^{2}}\left(q^{2m}-1\right)q^{m(2j+2i-2n+m-1)}\\
        &\qquad-\sum_{m=0}^{j-i-1}\chi_{i+1,j-m-1}(q)\gs{j-i}{m}{q^{2}}\left(q^{2(j-i)}-q^{2m}\right)q^{m(2j+2i-2n+m-1)}\\
		&\qquad+\sum_{m=0}^{j-i-1}\chi_{i+1,j-m-1}(q)\gs{j-i}{m}{q^{2}}\left(q^{2(2j-i-m)-1}-q^{2j-1}\right)q^{m(2j+2i-2n+m-1)}\\
        &\qquad+\left(q^{2j}-q^{2(j-i)}\right)q^{(j-i)(3j+i-2n-1)}\\
		&\qquad+\left(q^{2i}-1\right)q^{2(j-i)}\left. \sum_{m=0}^{j-i-1}\chi_{i+1,j-m}(q)\gs{j-i}{m}{q^{2}}q^{m(2j+2i-2n+m-1)}\right)\\
		&=q^{j(4n-5j)}\chi_{1,i}(q)\frac{q^{2j-2}}{q^2-1}\left(\left(q^{2(j-i)}-1\right)q^{(j-i)(3j+i-2n-1)}\vphantom{\sum_{m=0}^{j-i}\gs{j-i}{m}{q^{2}}}\right.\\
        &\qquad+\sum_{m=1}^{j-i-1}\chi_{i+1,j-m-1}(q)\gs{j-i}{m}{q^{2}}q^{m(2j+2i-2n+m-1)}\\&\qquad\qquad\left(\left(q^{2j-2m-1}-1\right)\left(q^{2m}-1\right)-\left(q^{2(j-i)}-q^{2m}\right)+\left(q^{2(2j-i-m)-1}-q^{2j-1}\right)\right)\\
        &\qquad-\chi_{i+1,j-1}(q)\left(q^{2(j-i)}-1\right)+\chi_{i+1,j-1}(q)\left(q^{2(2j-i)-1}-q^{2j-1}\right)\\&\qquad+\left(q^{2j}-q^{2(j-i)}\right)q^{(j-i)(3j+i-2n-1)}\\
		&\qquad+\left(q^{2j}-q^{2(j-i)}\right)\left. \sum_{m=0}^{j-i-1}\chi_{i+1,j-m}(q)\gs{j-i}{m}{q^{2}}q^{m(2j+2i-2n+m-1)}\right)\\
		&=q^{j(4n-5j)}\chi_{1,i}(q)\frac{q^{2j-2}}{q^2-1}\left(\left(q^{2j}-1\right)q^{(j-i)(3j+i-2n-1)}\vphantom{\sum_{m=0}^{j-i}\gs{j-i}{m}{q^{2}}}\right.\\
        &\qquad+\left(q^{2(j-i)}-1\right)\sum_{m=1}^{j-i-1}\chi_{i+1,j-m-1}(q)\left(q^{2j-2m-1}-1\right)\gs{j-i}{m}{q^{2}}q^{m(2j+2i-2n+m-1)}\\
        &\qquad+\chi_{i+1,j-1}(q)\left(q^{2(j-i)}-1\right)\left(q^{2j-1}-1\right)\\
		&\qquad+\left(q^{2j}-q^{2(j-i)}\right)\left. \sum_{m=0}^{j-i-1}\chi_{i+1,j-m}(q)\gs{j-i}{m}{q^{2}}q^{m(2j+2i-2n+m-1)}\right)\\
		&=q^{j(4n-5j)}\chi_{1,i}(q)\frac{q^{2j-2}}{q^2-1}\left(\left(q^{2j}-1\right)q^{(j-i)(3j+i-2n-1)}\vphantom{\sum_{m=0}^{j-i}\gs{j-i}{m}{q^{2}}}\right.\\
        &\qquad+\left(q^{2(j-i)}-1\right)\sum_{m=0}^{j-i-1}\chi_{i+1,j-m}(q)\gs{j-i}{m}{q^{2}}q^{m(2j+2i-2n+m-1)}\\
		&\qquad+\left(q^{2j}-q^{2(j-i)}\right)\left. \sum_{m=0}^{j-i-1}\chi_{i+1,j-m}(q)\gs{j-i}{m}{q^{2}}q^{m(2j+2i-2n+m-1)}\right)\\
		&=q^{j(4n-5j)}\chi_{1,i}(q)\frac{q^{2j-2}}{q^2-1}\left(\left(q^{2j}-1\right)\sum_{m=0}^{j-i}\chi_{i+1,j-m}(q)\gs{j-i}{m}{q^{2}}q^{m(2j+2i-2n+m-1)}\right)
	\end{align*}
	from which the formula for $\gamma_{2i,2j,2n}$ follows since the coefficient of $\gamma_{2i,2j,2n}$ in \eqref{eq:symplecticinduction} is non-zero:
	\[
		\beta_{0,2n-2j,2n}=q^{2j-2}\frac{q^{2j}-1}{q^{2}-1}\neq0\:.\qedhere
	\]
\end{proof}

\begin{theorem}\label{th:gammageneralsymplectic}
	If $i$ or $j$ is odd, then $\gamma_{i,j,2n,2k}=0$. For $0\leq i\leq \min\{j,n-j\}$ and $j+k\leq n$ we have that
	\begin{align*}
		\gamma_{2i,2j,2n,2k}&=\beta_{2i,2j,2n,2k+2j}\gamma_{2i,2j,2k+2j}\\
		&=q^{2(k+i)(n-k-j)+j(4k-j)}\gs{n-j-i}{k-i}{q^2}\chi_{1,i}(q)\sum_{m=0}^{j-i}\chi_{i+1,j-m}(q)\gs{j-i}{m}{q^{2}}q^{m(2i-2k+m-1)} .
	\end{align*}
\end{theorem}
\begin{proof}
	Consider a fixed $2i$-singular $2j$-space. Now, we count in two ways the pairs $(\sigma,\tau)$, where $\sigma$ is a non-singular $2k$-space disjoint from $\pi$, and $\tau$ is a non-singular $(2k+2j)$-space and $\pi,\sigma\subseteq\tau$. The equality $\gamma_{2i,2j,2n,2k}=\beta_{2i,2j,2n,2k+2j}\gamma_{2i,2j,2k+2j}$ immediately follows. The second part of the result then follows from Lemma \ref{lem:betasymplectic} and Theorem \ref{th:gammasymplectic}.
\end{proof}

\subsection{The proportion of non-singular trivially intersecting subspaces spanning a non-singular space}
In this subsection, we look at the proportion that motivated this research.

\begin{definition}\label{def:rhosymplectic}
    Given a non-degenerate symplectic form on $\F^{2n}_{q}$ and integers $j,k$ with $0\leq j,k\leq n-1$ and $j+k\leq n$, let $\mathcal{S}_{2j,2k}$ be the set of pairs $(\pi,\pi')$ with $\dim(\pi)=2j$ and $\dim(\pi')=2k$ and both $\pi$ and $\pi'$ non-singular. Let $\mathcal{T}_{2j,2k}$ be the subset of $\mathcal{S}_{2j,2k}$ with pairs $(\pi,\pi')$ such that $\dim(\langle\pi,\pi'\rangle)=2j+2k$ and $\langle\pi,\pi'\rangle$ non-singular. The proportion $\frac{|\mathcal{T}_{2j,2k}|}{|\mathcal{S}_{2j,2k}|}$ is denoted by $\rho_{2j,2k,2n}$.
\end{definition}

Note that by definition $\rho_{2j,2k,2n}=\rho_{2k,2j,2n}$.

\begin{theorem}\label{th:rhosymplectic}
	For integers $j,k,n$ with $0\leq j,k\leq n-1$ and $j+k\leq n$, we have
    \[
        \rho_{2j,2k,2n}=q^{j(2k-j)}\frac{\psi^{-}_{n-j-k+1,n-j}(q^2)}{\psi^{-}_{n-k+1,n}(q^2)}\sum_{m=0}^{j}\chi_{1,j-m}(q)\gs{j}{m}{q^{2}}q^{m(m-2k-1)}\;.
    \]
\end{theorem}
\begin{proof}
    Since the symplectic group $\PGSp(2n,q)$ acts transitively on the non-singular $2j$-spaces we have immediately that $\rho_{2j,2k,2n}=\frac{\gamma_{0,2j,2n,2k}}{\alpha_{0,2k,2n}}$. From Lemma \ref{lem:alphasymplectic} and Theorem \ref{th:gammageneralsymplectic} we get
    \begin{align*}
        \rho_{2j,2k,2n}&=\frac{q^{2k(n-k+j)-j^{2}}\gs{n-j}{k}{q^2}\sum_{m=0}^{j}\chi_{1,j-m}(q)\gs{j}{m}{q^{2}}q^{m(m-2k-1)}}{q^{2k(n-k)}\gs{n}{k}{q^2}}\\
        &=q^{j(2k-j)}\frac{\psi^{-}_{n-j-k+1,n-j}(q^2)}{\psi^{-}_{n-k+1,n}(q^2)}\sum_{m=0}^{j}\chi_{1,j-m}(q)\gs{j}{m}{q^{2}}q^{m(m-2k-1)}\;.\qedhere
    \end{align*}
\end{proof}

\begin{remark}
    In \cite{gnp} it was proven that $\rho_{2j,2k,2(j+k)}\geq 1-\frac{5}{3}q^{-1}$, which was improved to $\rho_{2j,2k,2(j+k)}\geq 1-\frac{10}{7}q^{-1}$ in \cite{gim}. In \cite{gnp2} it was proven that $\rho_{2j,2k,2n}\geq 1-\frac{7}{4}q^{-1}$ if $n\geq j+k+1$. The previous theorem improves these results by giving the exact value of the proportion.
\end{remark}

Although we have an exact expression for the ratio of non-singular pairs that are disjoint, we will present a lower bound for it, as to get a better feeling for its asymptotic behavior.
\par It is immediate (both from Definition \ref{def:rhosymplectic} and Theorem \ref{th:rhosymplectic}) that $\rho_{0,2k,2n}=1$, and by symmetry $\rho_{2j,0,2n}=1$. We now look at the other cases.

\begin{theorem}\label{th:symplecticrhobound}
    Let $j$ and $k$ be integers with $j,k\geq 1$. We have that
    \[
        \rho_{2j,2k,2(j+k)}\geq 1-\frac{1}{q}-\frac{1}{q^2}+\max\left\{\frac{1}{q^{2j+1}},\frac{1}{q^{2k+1}}\right\}
    \]
    and
    \[
        \rho_{2j,2k,2n}\geq 1-\frac{1}{q}-\frac{1}{q^3}
    \]
    for all $n\geq j+k+1$.
\end{theorem}
\begin{proof}
    Without loss of generality we can assume $j\leq k$. From Theorem \ref{th:rhosymplectic} we have immediately
    \begin{align}\label{eq:symplecticrhobound}
         \rho_{2j,2k,2n}&=q^{j(2k-j)}\frac{\psi^{-}_{n-j-k+1,n-j}(q^2)}{\psi^{-}_{n-k+1,n}(q^2)}\sum_{m=0}^{j}\chi_{1,j-m}(q)\gs{j}{m}{q^{2}}q^{m(m-2k-1)}\nonumber\\
         &=q^{j(2k-j)}\frac{\psi^{-}_{n-j-k+1,n-k}(q^2)}{\psi^{-}_{n-j+1,n}(q^2)}\sum_{m=0}^{j}\chi_{1,j-m}(q)\gs{j}{m}{q^{2}}q^{m(m-2k-1)}\nonumber\\
         &\geq q^{j(2k-j)} \prod^{j}_{\ell=1}\left(\frac{q^{2(n-j-k+\ell)}-1}{q^{2(n-j+\ell)}-1}\right)\chi_{1,j}(q)\nonumber\\
         &\geq q^{2jk-j^{2}} \prod^{j}_{\ell=1}\left(q^{-2k}-q^{-2(n-j+\ell)}\right)\chi_{1,j}(q)\nonumber\\
         &=q^{j(2k+2j-2n)-j^{2}-2\binom{j+1}{2}} \prod^{j}_{\ell=1}\left(q^{2(n-j-k+\ell)}-1\right)\chi_{1,j}(q)\;.
    \end{align}
    For $n=j+k$ Equation \eqref{eq:symplecticrhobound} simplifies to
    \begin{align*}
        \rho_{2j,2k,2(j+k)}&\geq q^{-2j^{2}-j} \prod^{j}_{\ell=1}\left(q^{2\ell}-1\right)\chi_{1,j}(q)=q^{-2j^{2}-j}\psi^{-}_{1,2j}(q)\;.
    \end{align*}
    Using Lemma \ref{lem:psiminbounds} we then find
    \begin{align*}
        \rho_{2j,2k,2(j+k)}&\geq q^{-2j^{2}-j}\psi^{-}_{1,2j}(q)\geq q^{-2j^{2}-j}q^{\binom{2j+1}{2}}\left(1-\frac{1}{q}-\frac{1}{q^2}+\frac{1}{q^{2j+1}}\right)\geq1-\frac{1}{q}-\frac{1}{q^2}+\frac{1}{q^{2j+1}}\;.
    \end{align*}
    For $n\geq j+k+1$ we know that $q^{2(n-j-k-1)}\geq1$ and thus
    \begin{align*}
        \rho_{2j,2k,2n}&\geq q^{j(2k+2j-2n)-j^{2}-2\binom{j+1}{2}} \prod^{j}_{\ell=1}\left(q^{2(n-j-k+\ell)}-q^{2(n-j-k-1)}\right)\chi_{1,j}(q)\\
        &=q^{-2j-j^{2}-2\binom{j+1}{2}} \prod^{j}_{\ell=1}\left(q^{2(\ell+1)}-1\right)\chi_{1,j}(q)\\
        &=q^{-2j^{2}-3j}\psi^{-}_{1,2j}(q)\frac{q^{2j+2}-1}{q^2-1}\\
        &\geq q^{-2j^{2}-3j}q^{\binom{2j+2}{2}-1}\left(1-\frac{1}{q}-\frac{1}{q^3}+\frac{q^2-q+1}{q^{2j+3}}\right)\\
        &\geq 1-\frac{1}{q}-\frac{1}{q^3}\;.\qedhere
    \end{align*}
\end{proof}

\begin{remark} We can deduce the following from Theorem \ref{th:symplecticrhobound}. Note that these bounds always improve on those found in $\cite{gnp}$ and \cite{gnp2}:
\begin{itemize}
\item if $n\geq j+k+1$, or $j=1$ or $k=1$, then $\rho_{2j,2k,2n}\geq 1-\frac{5}{4}q^{-1}$,
\item if $n=j+k$ it follows from the proof of Theorem \ref{th:symplecticrhobound} that $\rho_{2j,2k,2(j+k)}\geq q^{-2j^2-j}\psi^{-}_{1,2j}(q)=(1-\frac{1}{q})(1-\frac{1}{q^2})(1-\frac{1}{q^3})\ldots (1-\frac{1}{q^{2j}})$. This lower bound was already observed in \cite{gim}. From it, it
follows that $\rho_{2j,2k,2(j+k)}\geq 1-1.4224 q^{-1}\geq 1-\frac{10}{7} q^{-1}$.
\end{itemize}
\end{remark}

\begin{remark}
    The coefficients of $1$, $q^{-1}$ and $q^{-2}$ from Theorem \ref{th:symplecticrhobound} cannot be improved since some more detailed analysis shows that
    \begin{align*}
        \rho_{2j,2k,2n}=
        \begin{cases}
            1-q^{-1}-q^{-2}+O\left(q^{-3}\right) &\quad n=j+k\text{ and }1\in\{j,k\},\\
            1-q^{-1}-q^{-2}+O\left(q^{-5}\right) &\quad n=j+k\text{ and }j,k\geq2,\\
            1-q^{-1}+O\left(q^{-3}\right) &\quad n\geq j+k+1\text{ and }1\in\{j,k\},\\
            1-q^{-1}-q^{-3}+O\left(q^{-5}\right) &\quad n\geq j+k+1\text{ and }j,k\geq2.
        \end{cases}
    \end{align*}
\end{remark}

\paragraph{Acknowledgment:} The research of the first author is partially supported by the Croatian Science Foundation under the project 5713.

\end{document}